\documentclass[a4paper,12pt]{book}
\usepackage{graphicx,epsfig}
\usepackage{amsfonts}
\usepackage{amssymb}
\usepackage{amsmath}

\usepackage{times}
\usepackage{amsmath,amsthm}
\usepackage{graphics}

\newtheorem{theo}{Theorem}
\newtheorem{lem}{Lemma}
\newtheorem{cor}{Corollary}
\newtheorem{rem}{Remark}
\newtheorem{ex}{Example}
\newtheorem{defn}{Definition}
\def\RR{\mathbb R}
\def\CC{\mathbb C}

\def\pmatrix{ \left( \begin{array} }
\def\endpmatrix{ \end{array} \right) }
\def\sigmd{{\dot\sigma}}

\def\no{\noindent}

\def\d{{\rm d}}

\def\pmatrix{ \left( \begin{array} }
\def\endpmatrix{ \end{array} \right) }
\def\aa{\alpha}

\def\cc{\gamma}
\def\dd{\delta}

\def\lam{\lambda}

\def\bfp{\boldsymbol{p}}

\def\bfgam{\boldsymbol{\gamma}}
\def\bfdel{\boldsymbol{\delta}}

\def\bfpsi{\boldsymbol{\psi}}

\def\bfp{\boldsymbol{p}}
\def\bfq{\boldsymbol{q}}

\def\bfu{\boldsymbol{u}}
\def\bfv{\boldsymbol{v}}

\def\bfy{\boldsymbol{y}}

\def\bfo{\boldsymbol{0}}

\def\d2dxx{\frac{\partial^2}{\partial x^2}}

\def\no{\noindent}
\def\diag{{\rm diag}}
\def\proof{\underline{Proof}\quad}
\def\QED{~\mbox{$\Box$}}

\def\phi{\varphi}

\def\I{{\cal I}}
\def\P{{\cal P}}

\def\O{\Omega}
\def\dd{\mathrm{d}}

\title{\bf The Hamiltonian BVMs (HBVMs) Homepage\thanks{
Work developed within the project {\em Numerical Methods and Software for
Differential Equations}.}}

\author{Luigi Brugnano\thanks{Dipartimento di Matematica ``U.\,Dini'',
Universit\`a di Firenze, Italy \mbox{e-mail: {\tt luigi.brugnano@unifi.it}}}
\and Felice Iavernaro\thanks{Dipartimento di Matematica,
Universit\`a di Bari, Italy \mbox{e-mail: {\tt felix@dm.uniba.it}}} \and
Donato Trigiante\thanks{Dipartimento di Energetica ``S.\,Stecco'',
Universit\`a di Firenze, Italy, \mbox{e-mail: {\tt trigiant@unifi.it}}}}

\date{February 24, 2010}

\begin{document}

\maketitle

\chapter*{Preface}
{\em Hamiltonian Boundary Value Methods} (in short, {\em HBVMs})
is a new class of numerical methods for the efficient numerical
solution of canonical Hamiltonian systems. In particular, their
main feature  is that of exactly preserving, for the numerical
solution, the value of the Hamiltonian function, when the latter
is a polynomial of arbitrarily high degree.

Clearly, this fact implies a practical conservation of any
analytical Hamiltonian function.

In this notes, we collect the introductory material on HBVMs
contained in the {\em HBVMs Homepage}, available at the url:

\medskip
\centerline{\tt http://web.math.unifi.it/users/brugnano/HBVM/index.html}

\medskip
The notes are organized as follows:

\begin{itemize}
\item Chapter 1: Basic Facts about HBVMs

\item Chapter 2: Numerical Tests

\item Chapter 3: Infinity HBVMs

\item Chapter 4: Isospectral Property of HBVMs and their connections with Runge-Kutta
collocation methods

\item Chapter 5: Blended HBVMs

\item Chapter 6: Notes and References

\item Bibliography

\end{itemize}

\chapter{Basic Facts about HBVMs}\label{chap1}

We consider Hamiltonian problems in the form
\begin{equation}\label{hamilode}
\dot y(t) = J\nabla H(y(t)),  \qquad y(t_0) = y_0\in\RR^{2m},
\end{equation}

\no where $J$ is a skew-symmetric constant matrix, and the
Hamiltonian $H(y)$ is assumed to be sufficiently differentiable.
Usually,
$$J = \pmatrix{rr} & I_m\\ -I_m\endpmatrix, \qquad y =
\pmatrix{c}q\\p\endpmatrix, \quad q,p\in\RR^m,$$
so that (\ref{hamilode}) assumes the form
$$\dot q = \nabla_p H(q,p), \qquad \dot p = -\nabla_q H(q,p).$$

\no The induced dynamical system is characterized by the presence
of invariants of motion, among which the Hamiltonian itself:
$$\dot H(y(t)) = \nabla H(y(t))^T \dot y(t) = \nabla
H(y(t))^TJ\nabla H(y(t)) = 0,$$

\no due to the fact that $J$ is skew-symmetric. Such property is
usually lost, when numerically solving problem (\ref{hamilode}).
This drawback can be overcome by using Hamiltonian BVMs
(hereafter, HBVMs).

The key formula which HBVMs rely on, is the {\em line integral}
and the related property of conservative vector fields:
\begin{equation}\label{Hy}
H(y_1) - H(y_0) = h\int_0^1 \sigmd(t_0+\tau h)^T\nabla
H(\sigma(t_0+\tau h))\dd\tau,
\end{equation}

\no for any $y_1 \in \RR^{2m}$, where $\sigma$ is any smooth
function such that
\begin{equation}
\label{sigma}\sigma(t_0) = y_0, \qquad\sigma(t_0+h) =
y_1.
\end{equation}

\no Here we consider the case where $\sigma(t)$ is a polynomial of
degree $s$, yielding an approximation to the true solution $y(t)$
in the time interval $[t_0,t_0+h]$. The numerical approximation
for the subsequent time-step, $y_1$, is then defined by
(\ref{sigma}). After introducing a set of $s$ distinct abscissae
\begin{equation}\label{ci}0<c_{1},\ldots ,c_{s}\le1,\end{equation}

\no we set
\begin{equation}\label{Yi}Y_i=\sigma(t_0+c_i h), \qquad
i=1,\dots,s,\end{equation}

\no so that $\sigma(t)$ may be thought of as an interpolation
polynomial, interpolating the {\em fundamental stages} $Y_i$,
$i=1,\dots,s$. We observe that, due to (\ref{sigma}), $\sigma(t)$
also interpolates the initial condition $y_0$.

\begin{rem}\label{c0} Sometimes, the interpolation at $t_0$ is
explicitly  required. In such a case, the extra abscissa $c_0=0$
is formally added to (\ref{ci}). This is the case, for example, of
a Lobatto distribution of the abscissae
\cite{brugnano09bit}.\end{rem}

Let us consider the following expansions of
$\dot \sigma(t)$ and $\sigma(t)$ for $t\in [t_0,t_0+h]$:
\begin{equation}
\label{expan} \dot \sigma(t_0+\tau h) = \sum_{j=1}^{s} \gamma_j
P_j(\tau), \qquad \sigma(t_0+\tau h) = y_0 + h\sum_{j=1}^{s}
\gamma_j \int_{0}^\tau P_j(x)\,\dd x,
\end{equation}

\no where $\{P_j(t)\}$ is a suitable basis of the vector space of
polynomials of degree at most $s-1$ and  the  (vector)
coefficients $\{\gamma_j\}$ are to be determined. Because of the
arguments in \cite{brugnano09bit,BIT09,BIT10}, we shall consider
an {\bf orthonormal basis} of polynomials on the interval $[0,1]$,
i.e.:
\begin{equation}\label{orto}\int_0^1 P_i(t)P_j(t)\dd t = \delta_{ij}, \qquad
i,j=1,\dots,s,\end{equation}

\no where $\delta_{ij}$ is the Kronecker symbol, and $P_i(t)$ has
degree $i-1$. Such a basis can be readily obtained as
\begin{equation}\label{orto1}P_i(t) = \sqrt{2i-1}\,\hat P_{i-1}(t),
\qquad i=1,\dots,s,\end{equation} with $\hat P_{i-1}(t)$ the shifted
Legendre polynomial, of degree $i-1$, on the interval $[0,1]$.

\begin{rem}\label{recur}
 From the properties of shifted Legendre polynomials (see, e.g.,
\cite{AS} or the Appendix in \cite{brugnano09bit}), one readily
obtains that the polynomials $\{P_j(t)\}$ satisfy the three-terms
recurrence:
\begin{eqnarray*}
P_1(t)&\equiv& 1, \qquad P_2(t) = \sqrt{3}(2t-1),\\
P_{j+2}(t) &=& (2t-1)\frac{2j+1}{j+1} \sqrt{\frac{2j+3}{2j+1}} P_{j+1}(t)
-\frac{j}{j+1}\sqrt{\frac{2j+3}{2j-1}} P_j(t), \quad j\ge1.
\end{eqnarray*}
\end{rem}

We shall also assume that $H(y)$ is a polynomial, which implies
that the integrand in \eqref{Hy} is also a polynomial so that the
line integral can be exactly computed by means of a suitable
quadrature formula. In general, however, due to the high degree of
the integrand function,  such quadrature formula cannot be solely
based upon the available abscissae $\{c_i\}$: one needs to
introduce an additional set of abscissae $\{\hat c_1, \dots,\hat
c_r\}$, distinct from the nodes $\{c_i\}$, in order to make the
quadrature formula exact:
\begin{eqnarray} \label{discr_lin}
\displaystyle \lefteqn{\int_0^1 \sigmd(t_0+\tau h)^T\nabla
H(\sigma(t_0+\tau h))\dd\tau   =}\\ && \sum_{i=1}^s \beta_i
\sigmd(t_0+c_i h)^T\nabla H(\sigma(t_0+c_i h)) + \sum_{i=1}^r \hat
\beta_i \sigmd(t_0+\hat c_i h)^T\nabla H(\sigma(t_0+\hat c_i h)),
\nonumber
\end{eqnarray}

\no where $\beta_i$, $i=1,\dots,s$, and $\hat \beta_i$,
$i=1,\dots,r$, denote the weights of the quadrature formula
corresponding to the abscissae $\{c_i\}$ and $\{\hat c_i\}$,
respectively, i.e.,
\begin{eqnarray}\nonumber
\beta_i &=& \int_0^1\left(\prod_{ j=1,j\ne i}^s
\frac{t-c_j}{c_i-c_j}\right)\left(\prod_{j=1}^r
\frac{t-\hat c_j}{c_i-\hat c_j}\right)\dd t, \qquad i = 1,\dots,s,\\
\label{betai}\\ \nonumber \hat\beta_i &=& \int_0^1\left(\prod_{
j=1}^s \frac{t-c_j}{\hat c_i-c_j}\right)\left(\prod_{ j=1,j\ne
i}^r \frac{t-\hat c_j}{\hat c_i-\hat c_j}\right)\dd t, \qquad i =
1,\dots,r.
\end{eqnarray}

\begin{rem}\label{c01}
In the case considered in the previous Remark~\ref{c0}, i.e. when
$c_0=0$ is formally considered together with the abscissae
(\ref{ci}), the first product in each formula in (\ref{betai})
ranges from $j=0$ to $s$. Moreover, also the range of
$\{\beta_i\}$ becomes $i=0,1,\dots,s$. However, for sake of
simplicity, we shall not consider this case further.
\end{rem}

According to \cite{IT2}, the right-hand side of \eqref{discr_lin}
is called \textit{discrete line integral}, while the vectors
\begin{equation}\label{hYi}
\hat Y_i = \sigma(t_0+\hat c_i h), \qquad i=1,\dots,r,
\end{equation}

\no are called \textit{silent stages}: they just serve to
increase, as much as one likes, the degree of precision of the
quadrature formula, but they are not to be regarded as unknowns
since, from \eqref{expan}, they can be expressed in terms of
linear combinations of the \textit{fundamental stages} (\ref{Yi}).

\begin{defn}\label{defhbvmks}
The method defined by substituting the quantities in \eqref{expan}
into the right-hand side of \eqref{discr_lin}, and by choosing the
unknown coefficients $\{\gamma_j\}$ in order that the resulting
expression vanishes, is called {\em Hamiltonian Boundary Value
Method with $k$ steps and degree $s$}, in short {\em
HBVM($k$,$s$)}, where $k=s+r$ \, \cite{brugnano09bit}.\end{defn}

In such a way, one easily obtains, from (\ref{Hy})--(\ref{sigma}),
$$H(\sigma(t_0+h)) = H(y_0),$$

\no that is, the value of the Hamiltonian is {\em exactly}
preserved at the subsequent approximation, provided by
$\sigma(t_0+h)$.

In the sequel, we shall see that HBVMs may be expressed through
different, though equivalent, formulations: some of them can be
directly implemented in a computer program, the others being of
more theoretical interest.

Because of the equality \eqref{discr_lin}, we can apply the
procedure directly to the original line integral appearing in the
left-hand side. With this premise, by considering the first
expansion in \eqref{expan},  the conservation property  reads
\begin{equation}
\label{conservation} \sum_{j=1}^{s} \gamma_j^T \int_0^1  P_j(\tau)
\nabla H(\sigma(t_0+\tau h))\dd\tau=0,
\end{equation}

\no which, as is easily checked,  is certainly satisfied if we
impose the following set of orthogonality conditions
\begin{equation}
\label{orth} \gamma_j = \int_0^1  P_j(\tau) J \nabla
H(\sigma(t_0+\tau h))\dd\tau, \qquad j=1,\dots,s.
\end{equation}

\no Then, from the second relation of \eqref{expan} we obtain, by introducing
the operator
\begin{eqnarray}\label{Lf}\lefteqn{L(f;h)\sigma(t_0+ch) =}\\
\nonumber
&& \sigma(t_0)+h\sum_{j=1}^s \int_0^c P_j(x) \dd x \,
\int_0^1 P_j(\tau)f(\sigma(t_0+\tau h))\dd\tau,\qquad
c\in[0,1],\end{eqnarray}

\no that $\sigma$ is the eigenfunction of $L(J\nabla H;h)$
relative to the eigenvalue $\lambda=1$:
\begin{equation}\label{L}\sigma = L(J\nabla H;h)\sigma.\end{equation}

\begin{defn} Equation (\ref{L}) is the {\em Master Functional
Equation} defining $\sigma$ ~\cite{BIT09}.\end{defn}

\begin{rem}\label{MFE}
 From the previous arguments, one readily obtains that the Master
Functional Equation (\ref{L}) characterizes HBVM$(k,s)$ methods, for
all $k\ge1$. Indeed, such methods are uniquely defined by the
polynomial $\sigma$, of degree $s$, the number of steps $k$ being
only required to obtain an exact quadrature formula (see
(\ref{discr_lin})).\end{rem}

To practically compute $\sigma$, we set (see (\ref{Yi}) and
(\ref{expan}))
\begin{equation}
\label{y_i}
Y_i=  \sigma(t_0+c_i h) = y_0+ h\sum_{j=1}^{s} a_{ij}  \gamma_j, \qquad
i=1,\dots,s,
\end{equation}

\no where
\begin{equation}\label{aij}
a_{ij}=\int_{0}^{c_i} P_j(x) \dd x, \qquad
i,j=1,\dots,s.\end{equation}

\no Inserting \eqref{orth} into \eqref{y_i} yields the final
formulae which define the HBVMs class based upon the orthonormal
basis $\{P_j\}$:
\begin{equation}
\label{hbvm_int} Y_i=y_0+h\sum_{j=1}^s a_{ij}  \int_0^1 P_j(\tau)
J \nabla H(\sigma(t_0+\tau h))\dd\tau, \qquad i=1,\dots,s.
\end{equation}

For sake of completeness, we report the nonlinear system associated with the
HBVM$(k,s)$ method, in terms of the fundamental stages $\{Y_i\}$ and the silent
stages $\{\hat Y_i\}$ (see (\ref{hYi})), by using the notation
\begin{equation}\label{fy}
f(y) = J \nabla H(y).
\end{equation}

\no In this context, it  represents the discrete counterpart of
\eqref{hbvm_int}, and may be directly retrieved by evaluating, for
example, the integrals in \eqref{hbvm_int} by means of the (exact)
quadrature formula introduced in \eqref{discr_lin}:
\begin{eqnarray}\label{hbvm_sys}
\lefteqn{ Y_i =}\\
&& y_0+h\sum_{j=1}^s a_{ij}\left( \sum_{l=1}^s \beta_l
P_j(c_l)f(Y_l) + \sum_{l=1}^r\hat \beta_l P_j(\hat c_l) f(\widehat
Y_l) \right),\quad i=1,\dots,s.\nonumber
\end{eqnarray}

\no From the above discussion it is clear that, in the
non-polynomial case, supposing to choose the abscissae $\{\hat
c_i\}$ so that the sums in (\ref{hbvm_sys}) converge to an
integral as $r=k-s\rightarrow\infty$, the resulting formula is
\eqref{hbvm_int}. This implies that  HBVMs may be as well applied
in the non-polynomial case since, in finite precision arithmetic,
HBVMs are indistinguishable from their limit formulae
\eqref{hbvm_int}, when a sufficient number of silent stages is
introduced. The aspect of having a {\em practical} exact integral,
for $k$ large enough, was already stressed in \cite{BIS,
brugnano09bit, BIT09, IP1, IT2}.

We emphasize that, in the non-polynomial case, \eqref{hbvm_int}
becomes an operative method, only after that a suitable strategy
to approximate the integral is taken into account. In the present
case, if one discretizes the {\em Master Functional Equation}
(\ref{Lf})--(\ref{L}), HBVM$(k,s)$ are then obtained, essentially
by extending the discrete problem (\ref{hbvm_sys}) also to the
silent stages (\ref{hYi}). In order to simplify the exposition, we
shall use (\ref{fy}) and introduce the following notation:
\begin{eqnarray}\nonumber
\{\tau_i\} = \{c_i\} \cup \{\hat{c}_i\}, &&
\{\omega_i\}=\{\beta_i\}\cup\{\hat\beta_i\},\\
\label{tiyi}\\ \nonumber y_i = \sigma(t_0+\tau_ih), && f_i =
f(\sigma(t_0+\tau_ih)), \qquad i=1,\dots,k.
\end{eqnarray}

\no The discrete problem defining the HBVM$(k,s)$ then becomes,
\begin{equation}\label{hbvmks}
y_i = y_0 + h\sum_{j=1}^s \int_0^{\tau_i} P_j(x)\dd x
\sum_{\ell=1}^k \omega_\ell P_j(\tau_\ell)f_\ell, \qquad
i=1,\dots,k.
\end{equation}

\begin{rem}\label{ecc} We also observe that, from (\ref{orth}) and
the first relation in (\ref{expan}), one obtains the equations
\begin{equation}\label{ecceq} \dot\sigma(t_0+\tau_ih) = \sum_{j=1}^s
 P_j(\tau_i) \int_0^1 P_j(\tau)J\nabla H(\sigma(t_0+\tau
h))\dd\tau, \qquad i=1,\dots,k,
\end{equation}

\no which may be viewed as {\em extended collocation conditions}
according to \cite[Section\,2]{IT2}, where the integrals are
(exactly) replaced by discrete sums.
\end{rem}

By introducing the vectors $$\bfy = (y_1^T,\dots,y_k^T)^T, \qquad
e=(1,\dots,1)^T\in\RR^k,$$ and the matrices
\begin{equation}\label{OIP}\O=\diag(\omega_1,\dots,\omega_k), \qquad
\I_s,~\P_s\in\RR^{k\times s},\end{equation}

\no whose $(i,j)$th entry are given by
\begin{equation}\label{IDPO}
(\I_s)_{ij} = \int_0^{\tau_i} P_j(x)\dd x, \qquad
(\P_s)_{ij}=P_j(\tau_i), \end{equation}

\no we can cast the set of equations (\ref{hbvmks}) in vector form
as
\begin{equation}\label{rk0} \bfy = e\otimes y_0 + h(\I_s
\P_s^T\O)\otimes I_{2m}\, f(\bfy),\end{equation}

\no with an obvious meaning of $f(\bfy)$. Consequently, the method
can be seen as a Runge-Kutta method with the following Butcher
tableau:
\begin{equation}\label{rk}
\begin{array}{c|c}\begin{array}{c} \tau_1\\ \vdots\\ \tau_k\end{array} & \I_s \P_s^T\O\\
 \hline                    &\omega_1\, \dots~ \omega_k
                    \end{array}\end{equation}

\begin{rem}\label{ascisse} We observe that, because of the use of an
orthonormal basis, the role of the abscissae $\{c_i\}$ and of the
silent abscissae $\{\hat c_i\}$ is interchangeable, within the set
$\{\tau_i\}$. This is due to the fact that all the matrices
$\I_s$, $\P_s$, and $\O$ depend on all the abscissae $\{\tau_i\}$,
and not on a subset of them and, moreover, they are invariant with
respect to the choice of the fundamental abscissae $\{c_i\}$.
\end{rem}

The following result then holds true.

\begin{theo}\label{ordine} Provided that the quadrature defined by
the weights $\{\omega_i\}$ has order at least $2s$ (i.e., it is
exact for polynomials of degree at least $2s-1$), HBVM($k$,$s$)
has order $p=2 s\equiv 2\deg(\sigma)$, whatever the choice of the
abscissae $c_1,\dots,c_s$.
\end{theo}

\begin{proof} From the classical result of Butcher (see, e.g.,
\cite[Theorem\,7.4]{HNW}), the thesis follows if the usual simplifying
assumptions $C(s)$, $B(p)$, $p\ge 2s$, and $D(s-1)$ are satisfied
for the Runge-Kutta method defined by the tableau (\ref{rk}). By
looking at the method (\ref{rk0})--(\ref{rk}), one has that the
first two (i.e., $C(s)$ and $B(p)$, $p\ge 2s$) are obviously
fulfilled: the former by the definition of the method, the second
by hypothesis. The proof is then completed, if we prove $D(s-1)$.
Such condition can be cast in matrix form, by introducing the
vector $\bar{e}=(1,\dots,1)^T\in\RR^{s-1}$, and the matrices
$$Q=\diag(1,\dots,s-1),\qquad D=\diag(\tau_1,\dots,\tau_k),\qquad
V=(\tau_i^{j-1})\in\RR^{k\times s-1},$$

\no (see also (\ref{IDPO})) as
$$Q V^T\O\left(\I_s\P_s^T\O\right) = \left(\bar{e}\,e^T -V^TD\right)\O,$$ i.e.,
\begin{equation}\label{finito}
\P_s\I_s^T\O V Q = \left(e\,\bar{e}^T -DV\right).
\end{equation}

\no Since the quadrature is exact for polynomials of degree
$2s-1$, one has
\begin{eqnarray*}
\left(\I_s^T\O VQ\right)_{ij} &=& \left( \sum_{\ell=1}^k
\omega_\ell \int_0^{\tau_\ell} P_i(x)\dd x\,(j \tau_\ell^{j-1})
\right) = \left(\int_0^1 \, \int_0^t P_i(x)\dd x (jt^{j-1})\dd
t\right)
\\&=& \left( \delta_{i1}-\int_0^1P_i(x)x^j\dd x\right),
\qquad i = 1,\dots,s,\quad j=1,\dots,s-1,\end{eqnarray*}

\no where the last equality is obtained by integrating by parts,
with $\delta_{i1}$ the Kronecker symbol. Consequently,
\begin{eqnarray*}\left(\P_s\I_s^T \O V Q\right)_{ij} &=& \left(1 - \sum_{\ell=1}^s
P_\ell(\tau_i)\int_0^1 P_\ell(x) x^j\dd x \right)\\
&=& (1-\tau_i^j), \qquad i=1,\dots,k,\quad
j=1,\dots,s-1,\end{eqnarray*}

\no that is, (\ref{finito}), where the last equality follows from
the fact that
$$\sum_{\ell=1}^s P_\ell(\tau)\int_0^1 P_\ell(x)
x^j\dd x = \tau^j, \qquad j=1,\dots,s-1.\QED$$
\end{proof}

\medskip Concerning the stability of the methods, the following result holds true.

\begin{theo}\label{stab} For all $k$ such that the quadrature
formula has order at least $2s\equiv 2\deg(\sigma)$, HBVM($k$,$s$)
is perfectly $A$-stable,\footnote{That is, its region of Absolute
stability precisely coincides with the left-half complex plane,
$\CC^-$.} whatever the choice of the abscissae $c_1,\dots,c_s$.
\end{theo}

\begin{proof} As it has been previously observed, a HBVM$(k,s)$ is
fully characterized by the corresponding polynomial $\sigma$
which, for $k$ sufficiently large (i.e., assuming that
(\ref{discr_lin}) holds true), satisfies the {\em Master
Functional Equation} (\ref{Lf})--(\ref{L}), which is independent
of the choice of the nodes $c_1,\dots,c_s$ (since we consider an
orthonormal basis). When, in place of $f(y)=J\nabla H(y)$ we put the
test equation $f(y)=\lambda y$, we have that the collocation
polynomial of the Gauss-Legendre method of order $2s$, say
$\sigma_s$, satisfies the {\em Master Functional Equation}, since
the integrands appearing in it are polynomials of degree at most
$2s-1$, so that $\sigma=\sigma_s$. The proof completes by
considering that Gauss-Legendre methods are perfectly
$A$-stable.\QED
\end{proof}

\begin{ex}\label{lobex} As an example, for the methods studied in
\cite{brugnano09bit}, based on a Lobatto distribution of the nodes
$\{c_0=0,c_1,\dots,c_s\}\cup\{\hat{c}_1,\dots,\hat{c}_{k-s}\}$,
one has that $\deg(\sigma)=s$, so that the order of HBVM($k$,$s$)
turns out to be $2s$, with a quadrature satisfying $B(2k)$. Finally, we
observe that, with such choice of the abscissae HBVM$(s,s)$ reduces to
the Lobatto IIIA method of order $2s$.\end{ex}

\begin{ex}\label{gaussex} For the same reason, when one considers a Gauss
distribution for the abscissae
$\{c_1,\dots,c_s\}\cup\{\hat{c}_1,\dots,\hat{c}_{k-s}\}$, as done
in \cite{BIT09}, one also obtains a method of order $2s$ with a
quadrature satisfying $B(2k)$. Similarly as in the previous
example, HBVM$(s,s)$ now reduces to the Gauss-Legendre method of
order $2s$.
\end{ex}

\begin{rem}\label{symrem}
A number of remarks are in order, to emphasize relevant features
of HBVM$(k,s)$:
\begin{itemize}

\item From Remark~\ref{ascisse}, HBVM($k$,$s$) are {\em symmetric
methods} according to the {\em time reversal symmetry condition}
defined in \cite[p.\,218]{BT} (see also \cite{BT09}), provided
that the abscissae $\{\tau_i\}$ (see (\ref{tiyi})) are
symmetrically distributed ~\cite{brugnano09bit}.

\item By virtue of Theorems~\ref{ordine} and \ref{stab}, all
methods in Examples~\ref{lobex} and \ref{gaussex} are symmetric,
perfectly $A$-stable, and of order $2s$. In particular such
HBVM$(k,s)$ are exact for polynomial Hamiltonian functions of
degree $\nu$, provided that
\begin{equation}\label{knu}
k\ge \frac{\nu s}2.\end{equation}

\item For all $k$ sufficiently large so that (\ref{discr_lin}) holds,
HBVM$(k,s)$ based on the $k$ Gauss-Legendre abscissae in $[0,1]$
are {\em equivalent} to HBVM$(k,s)$ based on $k+1$ Lobatto
abscissae in $[0,1]$ (see \cite{BIT09}), since both methods define
the same polynomial $\sigma$ of degree $s$ (i.e., they satisfy the
same Master Functional Equation (\ref{L})--(\ref{Lf})).

\end{itemize}
\end{rem}

\chapter{Numerical Tests}\label{chap2}

We here collect a few numerical tests, in order to put into
evidence the potentialities of HBVMs
\cite{BIS1,brugnano09bit,BIT09}.

\section*{Test problem 1}  Let us consider the problem
characterized by the polynomial Hamiltonian (4.1) in \cite{Faou},
\begin{equation}\label{fhp}
H(p,q) = \frac{p^3}3 -\frac{p}2 +\frac{q^6}{30} +\frac{q^4}4
-\frac{q^3}3 +\frac{1}6,
\end{equation}

\no having degree $\nu=6$, starting at the initial point
$y_0\equiv (q(0),p(0))^T=(0,1)^T$, so that $H(y_0)=0$. For such a
problem, in \cite{Faou} it has been  experienced a numerical drift
in the discrete Hamiltonian, when using the fourth-order Lobatto
IIIA method with stepsize $h=0.16$, as confirmed by the plot in
Figure~\ref{faoufig0}. When using the fourth-order Gauss-Legendre
method the drift disappears, even though the Hamiltonian is not
exactly preserved along the discrete solution, as is confirmed by
the plot in Figure~\ref{faoufig}. On the other hand, by using the
fourth-order HBVM(6,2) with the same stepsize, the Hamiltonian
turns out to be preserved up to machine precision, as shown in
Figure~\ref{faoufig1}, since such method exactly preserves
polynomial Hamiltonians of degree up to 6. In such a case,
according to the last item in Remark~\ref{symrem}, the numerical
solutions obtained by using the Lobatto nodes
$\{c_0=0,c_1,\dots,c_6=1\}$ or the Gauss-Legendre nodes
$\{c_1,\dots,c_6\}$ are the same. The fourth-order convergence of
the method is numerically verified by the results listed in
Table~\ref{tp1}.

\begin{figure}[hp]
\centerline{\includegraphics[width=0.7\textwidth,height=6cm]{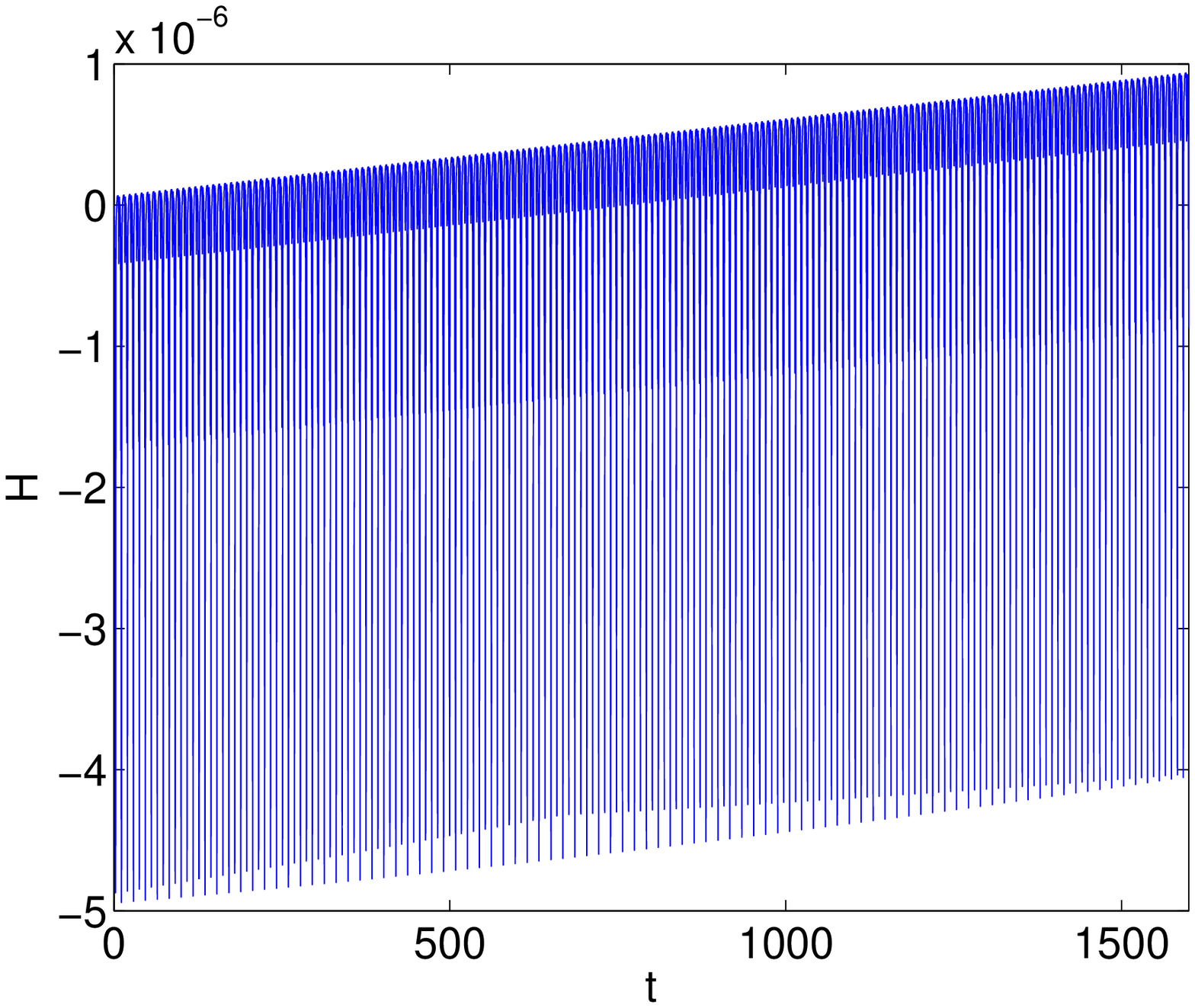}}
\caption{\protect\label{faoufig0} Fourth-order Lobatto IIIA
method, $h=0.16$, problem (\ref{fhp}): drift in the Hamiltonian.}

\centerline{\includegraphics[width=0.7\textwidth,height=6cm]{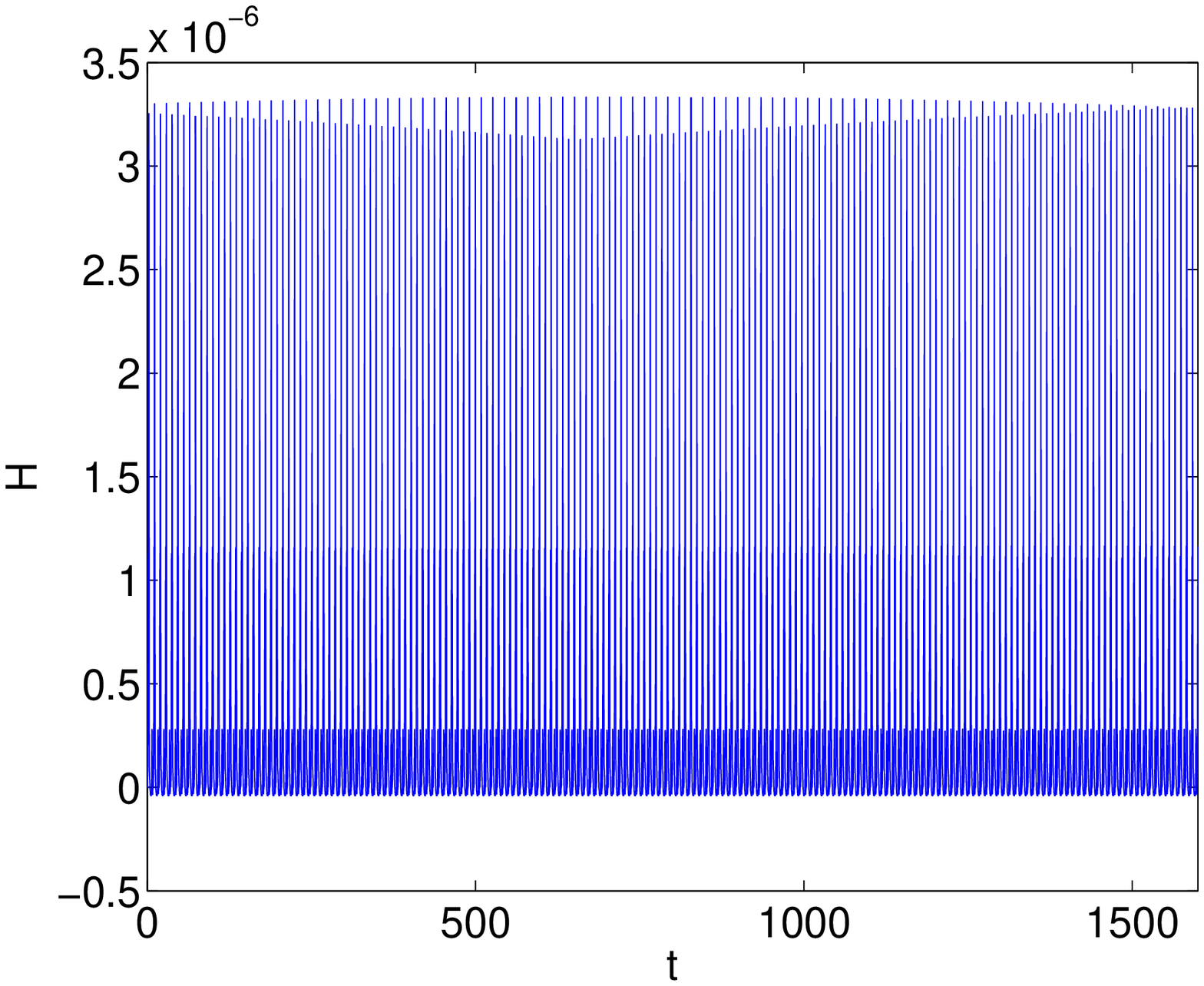}}
\caption{\protect\label{faoufig} Fourth-order Gauss-Legendre
method, $h=0.16$, problem (\ref{fhp}): $H\approx 10^{-6}$.}

\centerline{\includegraphics[width=0.7\textwidth,height=6cm]{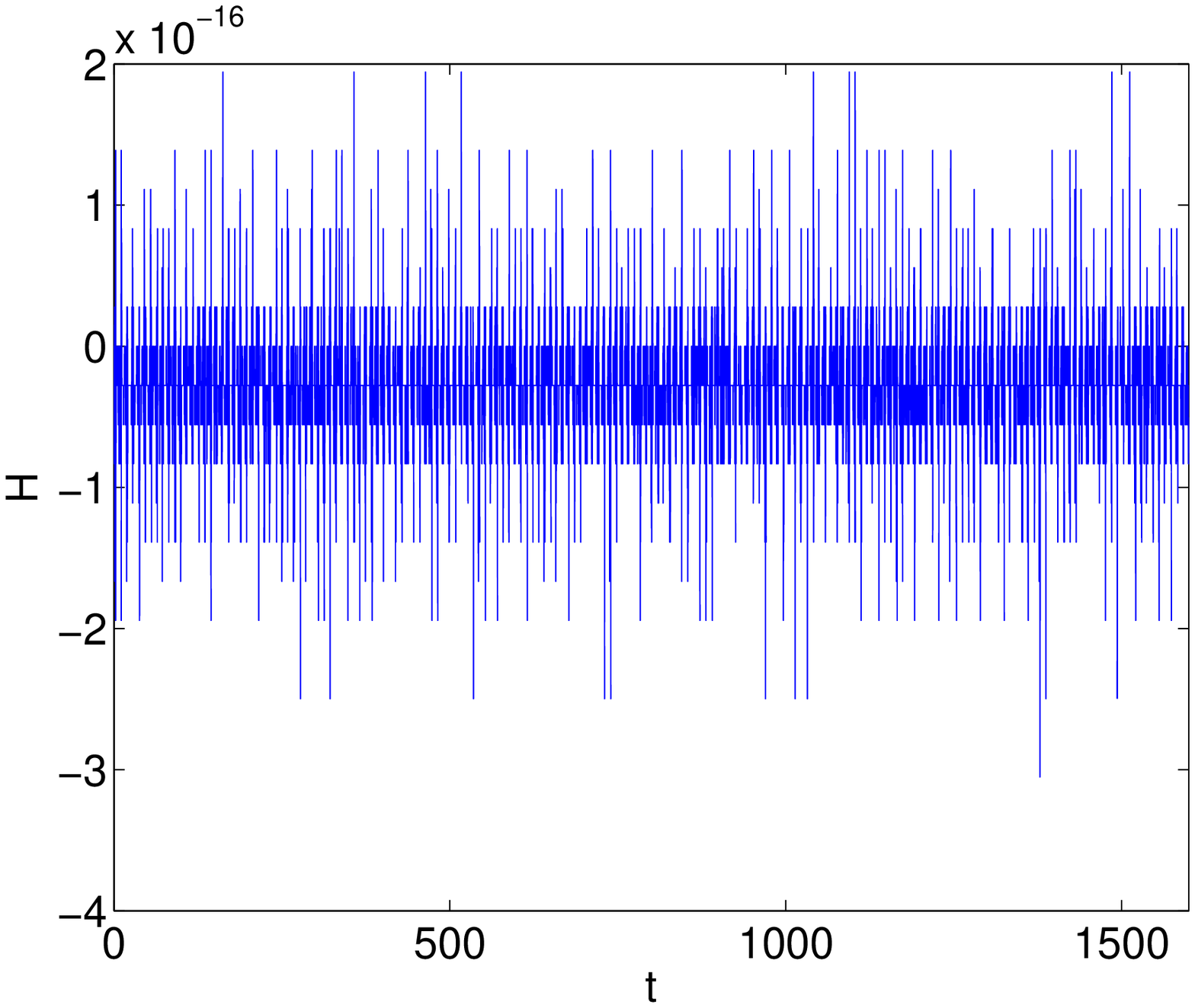}}
\caption{\protect\label{faoufig1} Fourth-order HBVM(6,2) method,
$h=0.16$, problem (\ref{fhp}): $H\approx 10^{-16}$.}
\end{figure}

\newpage
\section*{Test problem 2} The second test problem, having a
highly oscillating solution, is the Fermi-Pasta-Ulam problem (see
\cite[Section\,I.5.1]{hairer06gni}), modelling a chain of 2$m$ mass points
connected with alternating soft nonlinear and stiff linear springs, and fixed
at the end points. The variables $q_1, . . . , q_{2m}$ stand for the
displacements of the mass points, and $p_i = \dot q_i$ for their velocities. The
corresponding Hamiltonian, representing the total energy, is
\begin{equation}\label{fpu}
H(p,q) = \frac{1}2\sum_{i=1}^m\left(p_{2i-1}^2+p_{2i}^2\right)
+\frac{\omega^2}4\sum_{i=1}^m\left(q_{2i}-q_{2i-1}\right)^2
+\sum_{i=0}^m\left(q_{2i+1}-q_{2i}\right)^4,
\end{equation}

\no with $q_0=q_{2m+1}=0$. In our simulation we have used the following values:
$m=3$, $\omega=50$, and starting vector $$p_i=0, \quad q_i = (i-1)/10, \qquad
i=1,\dots,6.$$

\no In such a case, the Hamiltonian function is a polynomial of
degree 4, so that the fourth-order HBVM(4,2) method, either when
using the Lobatto nodes or the Gauss-Legendre nodes, is able to
exactly preserve the Hamiltonian, as confirmed by the plot in
Figure~\ref{fpufig1}, obtained with stepsize $h=0.05$. Conversely,
by using the same stepsize, both the fourth-order Lobatto IIIA and
Gauss-Legendre methods provide only an approximate conservation of
the Hamiltonian, as shown in the plots in Figures~\ref{fpufig0}
and \ref{fpufig}, respectively. The fourth-order convergence of
the HBVM(4,2) method is numerically verified by the results listed
in Table~\ref{tp2}.

\begin{figure}[hp]
\centerline{\includegraphics[width=0.7\textwidth,height=6cm]{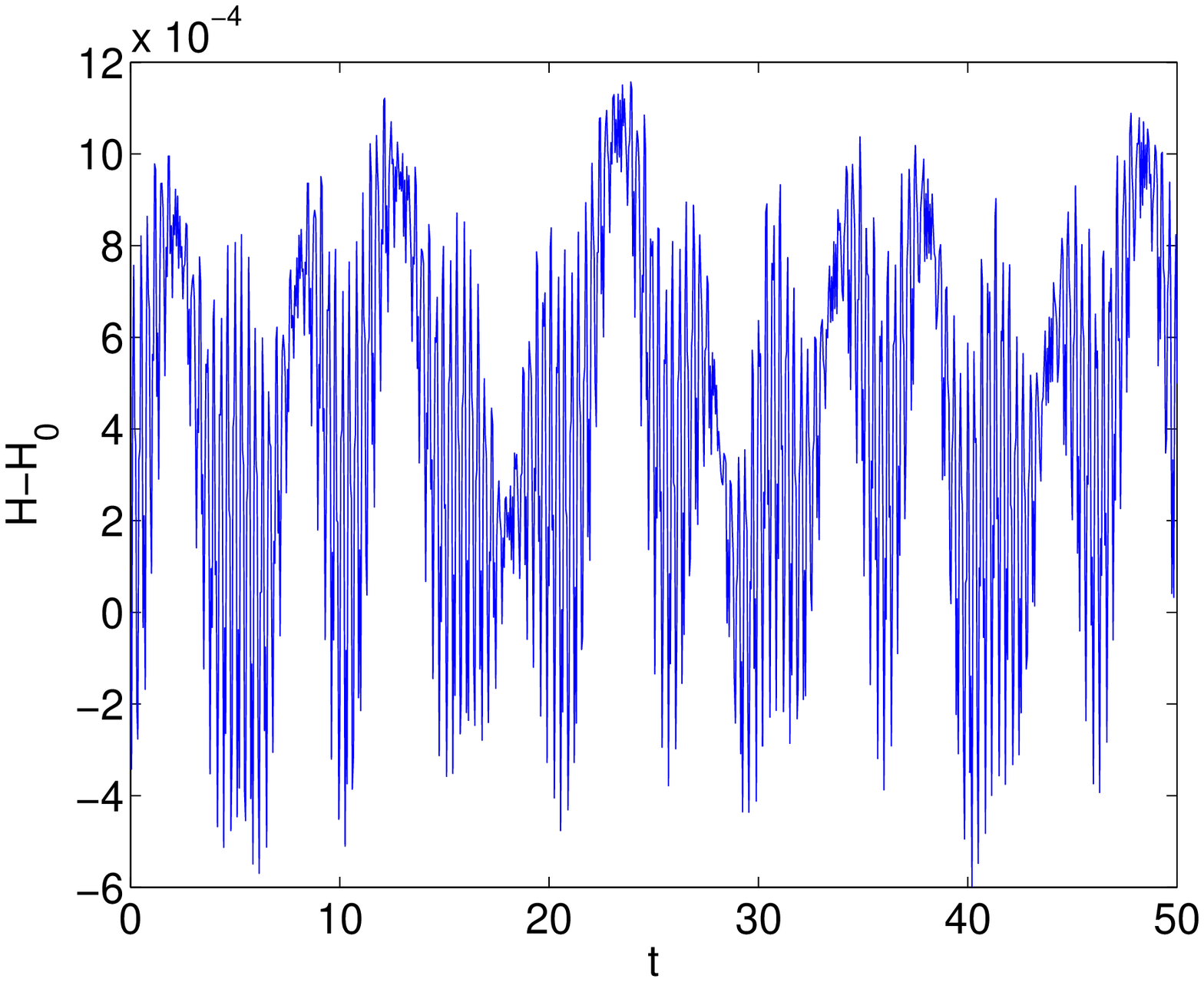}}
\caption{\protect\label{fpufig0} Fourth-order Lobatto IIIA method,
$h=0.05$, problem (\ref{fpu}): $|H-H_0|\approx 10^{-3}$.}

\centerline{\includegraphics[width=0.7\textwidth,height=6cm]{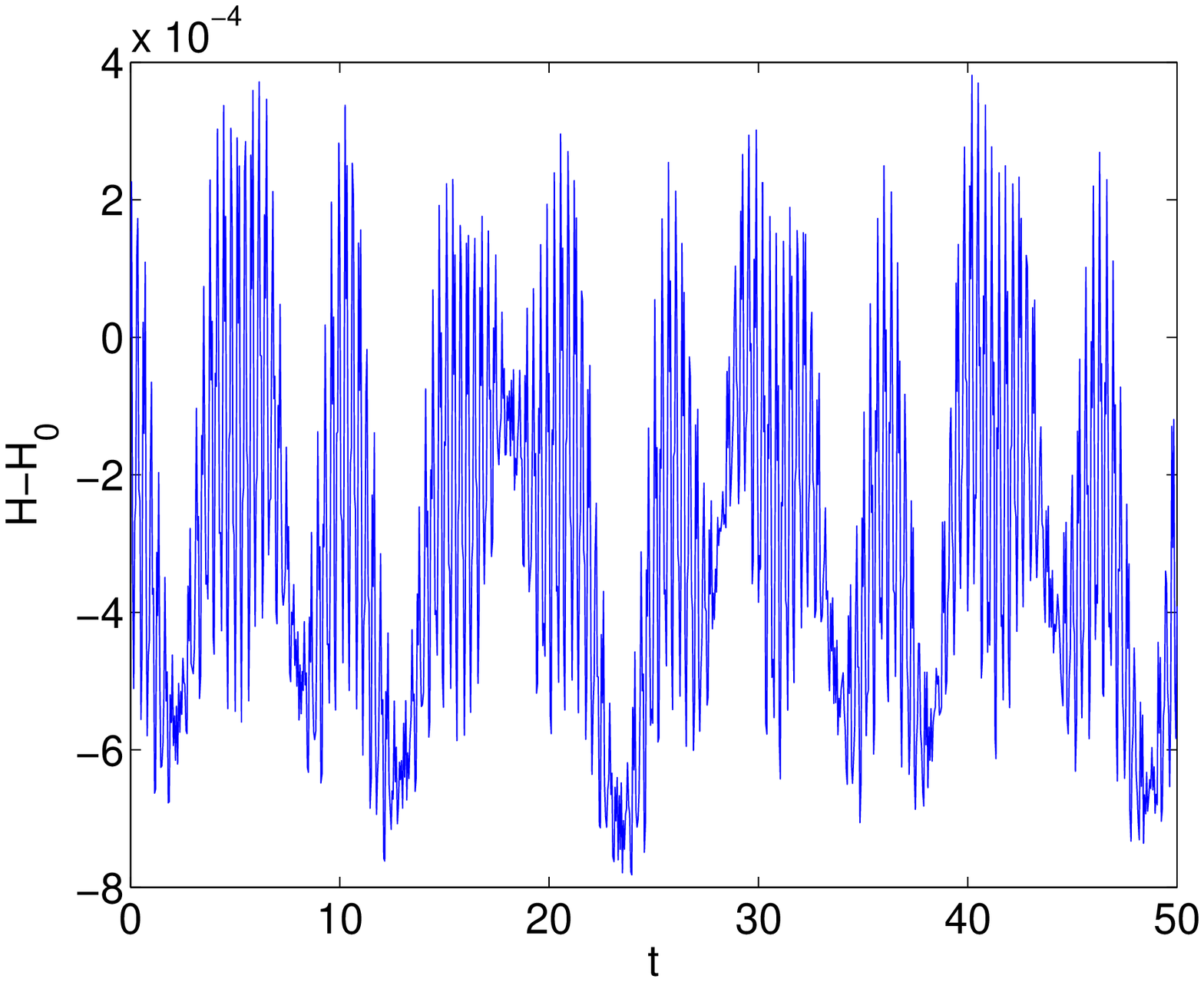}}
\caption{\protect\label{fpufig} Fourth-order Gauss-Legendre
method, $h=0.05$, problem (\ref{fpu}): $|H-H_0|\approx 10^{-3}$.}

\centerline{\includegraphics[width=0.7\textwidth,height=6cm]{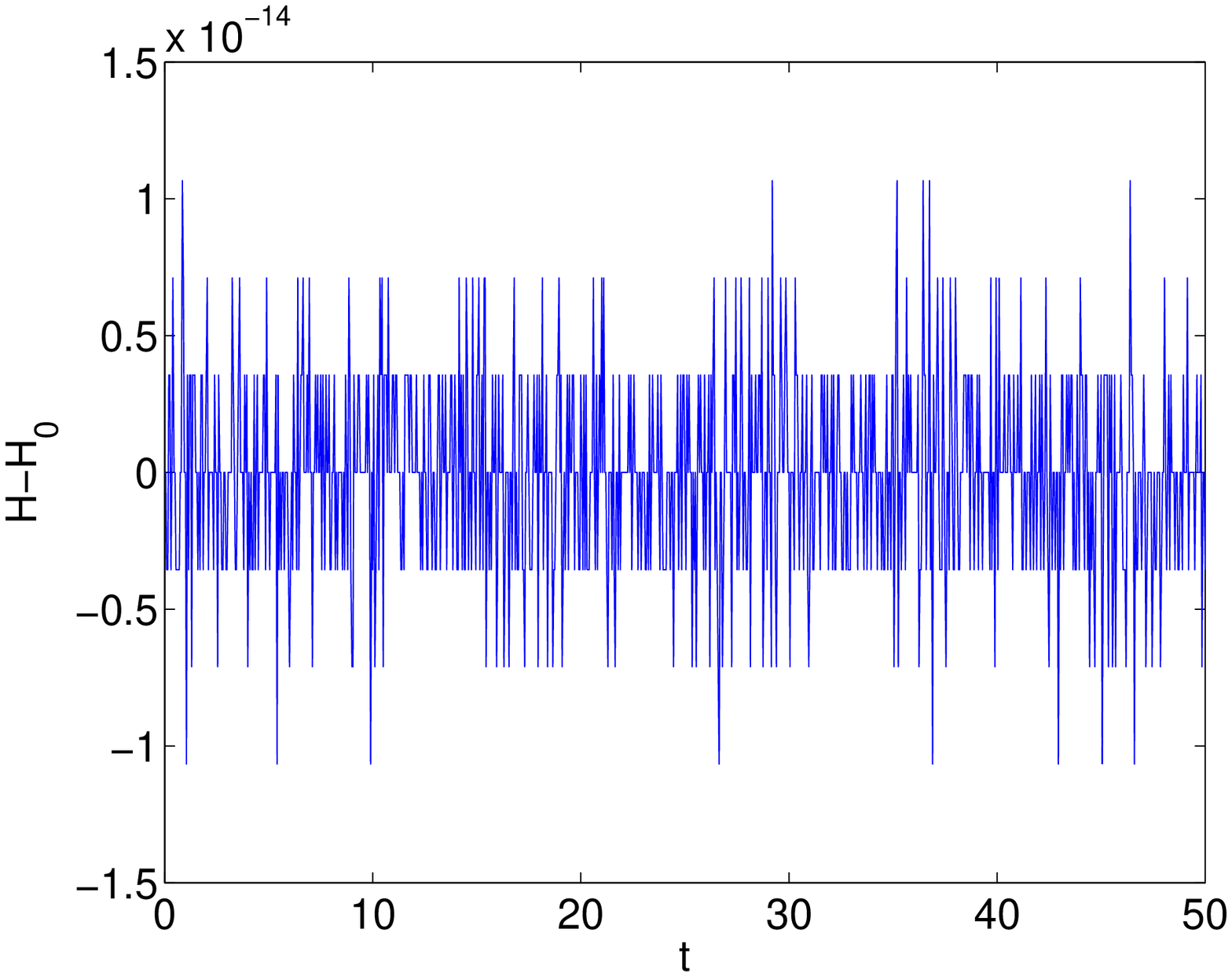}}
\caption{\protect\label{fpufig1} Fourth-order HBVM(4,2) method,
$h=0.05$, problem (\ref{fpu}): $|H-H_0|\approx 10^{-14}$.}
\end{figure}

\newpage

\section*{Test problem 3 (non-polynomial Hamiltonian)} In the
previous examples, the Hamiltonian function was a polynomial.
Nevertheless, as observed above, also in this case HBVM($k$,$s$)
are expected to produce  a {\em practical} conservation of the
energy when applied to systems defined by a non-polynomial
Hamiltonian function that can be locally well approximated by a
polynomial. As an example, we consider the motion of a charged
particle in a magnetic field with Biot-Savart
potential.\footnote{\,This kind of motion causes the well known
phenomenon of {\em aurora borealis}.} It is defined by the
Hamiltonian \cite{brugnano09bit}
\begin{eqnarray}\label{biot}
\lefteqn{H(x,y,z,\dot{x},\dot{y},\dot{z}) = }\\&&\frac{1}{2m}
\left[ \left(\dot{x}-\aa\frac{x}{\varrho^2}\right)^2 +
\left(\dot{y}-\aa\frac{y}{\varrho^2}\right)^2 +
\left(\dot{z}+\aa\log(\varrho)\right)^2\right],\nonumber
\end{eqnarray}

\no with $\varrho=\sqrt{x^2+y^2}$, $\aa= e \,B_0$,  $m$ is the
particle mass, $e$ is its charge, and $B_0$ is the magnetic field
intensity. We have used the values $$m=1, \qquad e=-1, \qquad
B_0=1,$$with starting point
$$x = 0.5, \quad y = 10, \quad z = 0, \quad
\dot{x} =  -0.1, \quad \dot{y} = -0.3, \quad \dot{z} = 0.$$

\no By using the fourth-order Lobatto IIIA method, with stepsize
$h=0.1$, a drift is again experienced in the numerical solution,
as is shown in Figure~\ref{biotfig0}. By using the fourth-order
Gauss-Legendre method with the same stepsize, the drift disappears
even though, as shown in Figure~\ref{biotfig1}, the value of the
Hamiltonian is preserved within an error of the order of
$10^{-3}$. On the other hand, when using the HBVM(6,2) method with
the same stepsize, the error in the Hamiltonian decreases to an
order of $10^{-15}$ (see Figure~\ref{biotfig2}), thus giving a
practical conservation. Finally, in Table~\ref{tab1} we list the
maximum absolute difference between the numerical solutions over
$10^3$ integration steps, computed by the HBVM$(k,2)$ methods
based on Lobatto abscissae and on Gauss-Legendre abscissae, as $k$
grows, with stepsize $h=0.1$. We observe that the difference tends
to 0, as $k$ increases. Finally, also in this case, one verifies a
fourth-order convergence, as the results listed in Table~\ref{tp3}
show.

\begin{figure}[hp]
\centerline{\includegraphics[width=0.7\textwidth,height=6cm]{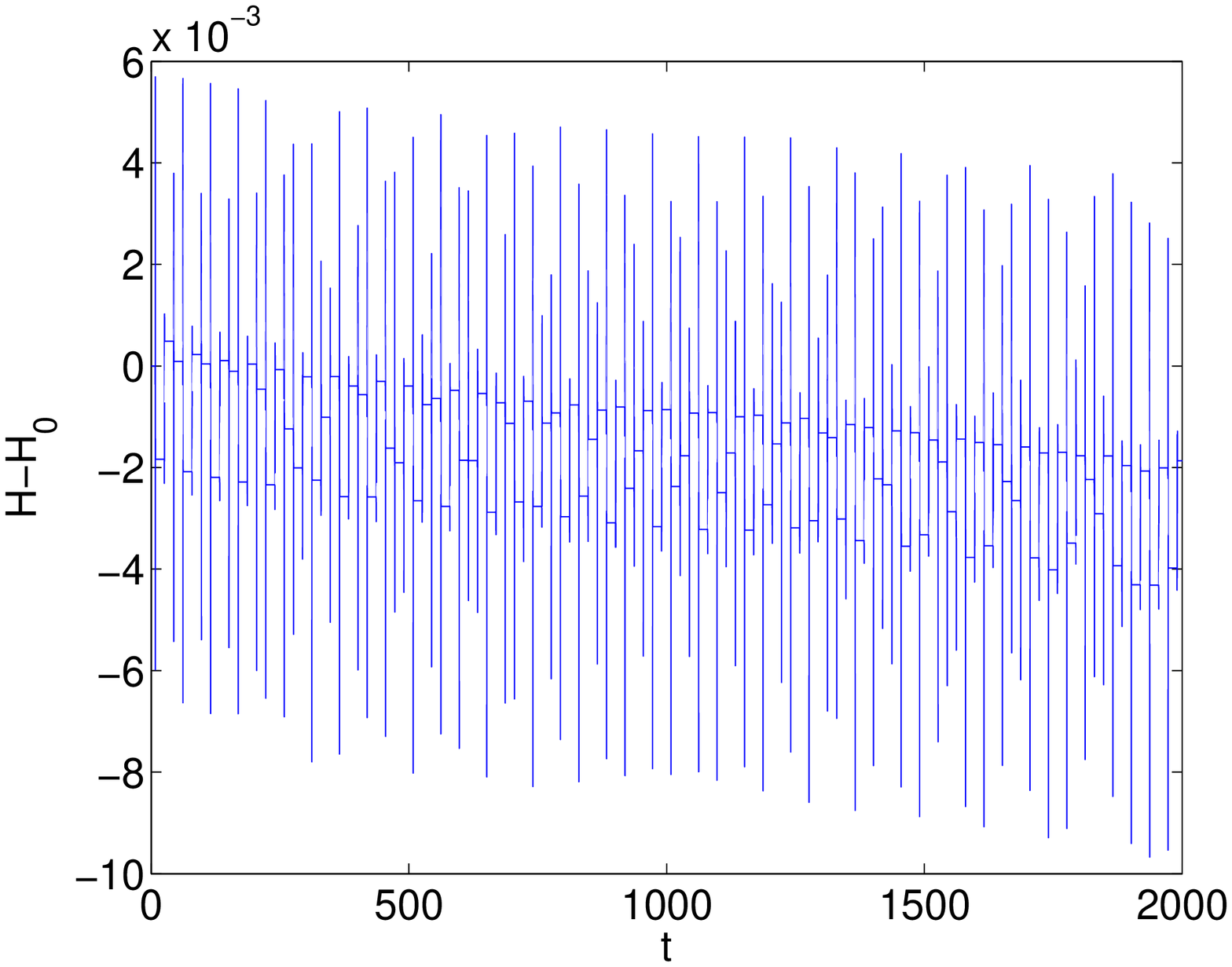}}
\caption{\protect\label{biotfig0} Fourth-order Lobatto IIIA
method, $h=0.1$, problem (\ref{biot}): drift in the Hamiltonian.}

\centerline{\includegraphics[width=0.7\textwidth,height=6cm]{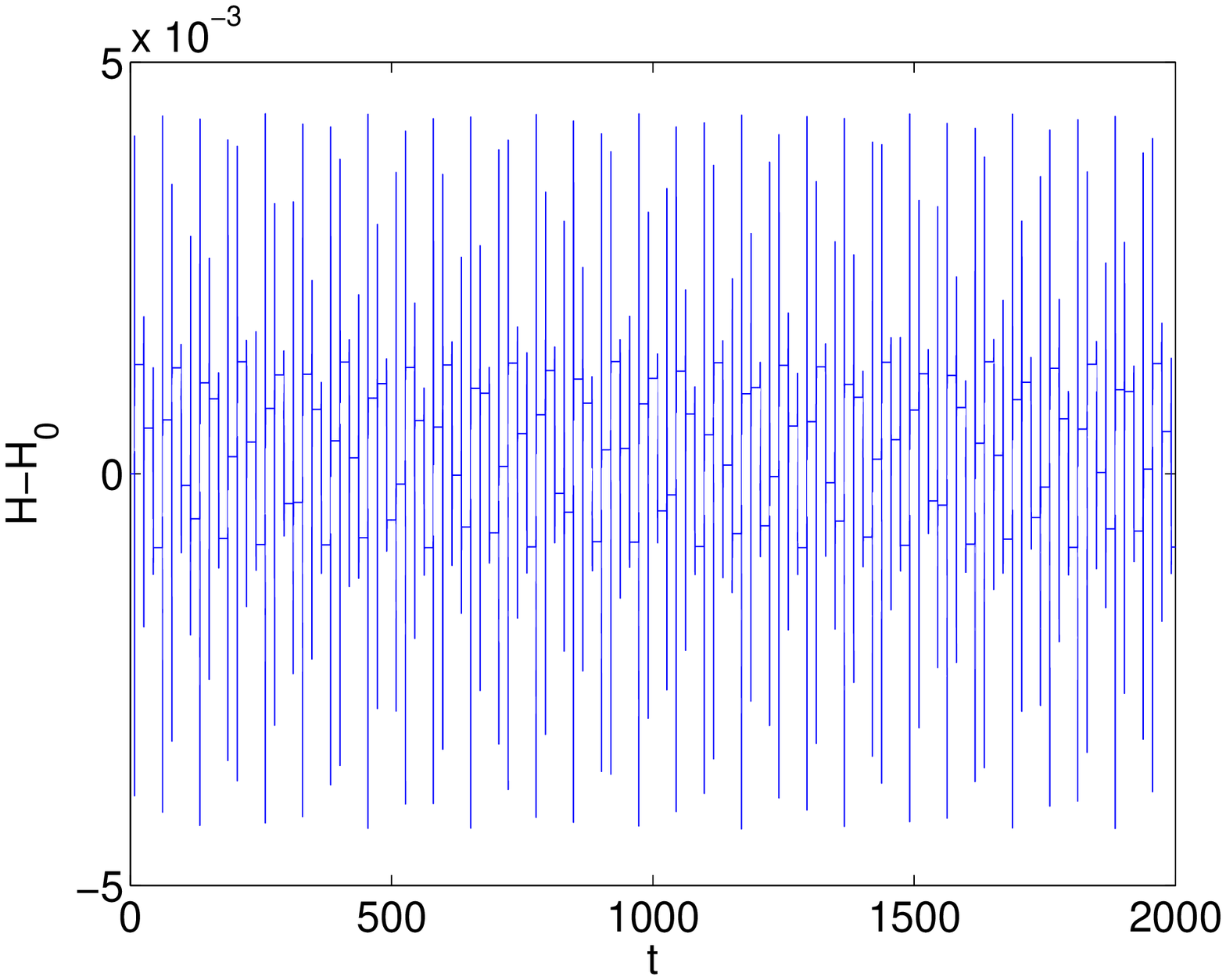}}
\caption{\protect\label{biotfig1} Fourth-order Gauss-Legendre
method, $h=0.1$, problem (\ref{biot}): $|H-H_0|\approx
10^{-3}$.}

\centerline{\includegraphics[width=0.7\textwidth,height=6cm]{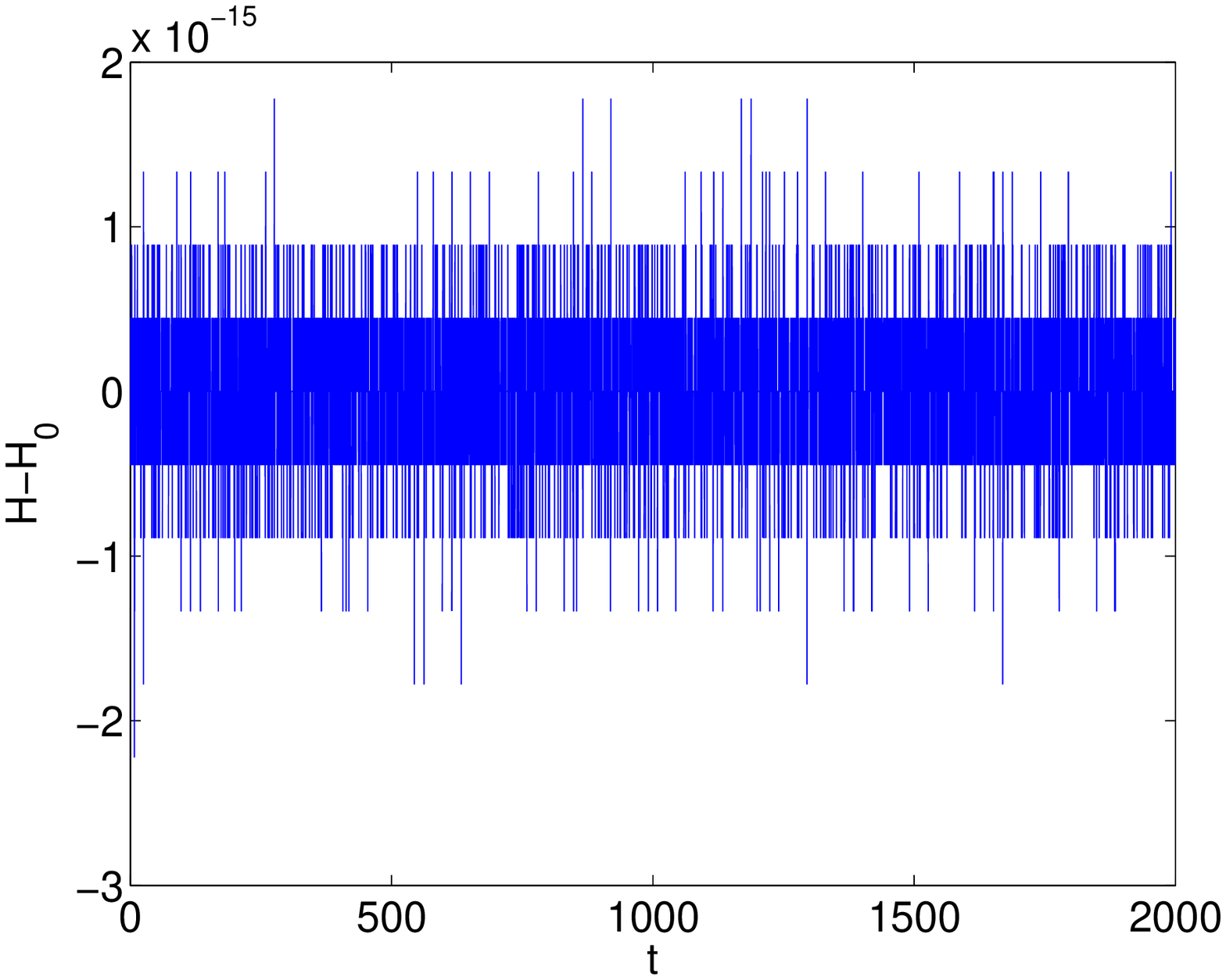}}
\caption{\protect\label{biotfig2} Fourth-order HBVM(6,2) method,
$h=0.1$, problem (\ref{biot}): $|H-H_0|\approx 10^{-15}$.}
\end{figure}

\begin{table}[hp]
\caption{\protect\label{tp1} Numerical order of convergence for
the HBVM(6,2) method, problem (\ref{fhp}).}
\begin{tabular}{|c|lllll|} \hline
$h$  & 0.32 & 0.16 & 0.08 & 0.04 & 0.02 \\
\hline
error  &$2.288\cdot10^{-2}$ &$1.487\cdot10^{-3}$ &$9.398\cdot10^{-5}$
&$5.890\cdot10^{-6}$ &$3.684\cdot10^{-7}$\\
\hline
order  & -- & 3.94 &3.98 &4.00 &4.00\\
\hline
\end{tabular}
\bigskip

\caption{\protect\label{tp2} Numerical order of convergence for
the HBVM(4,2) method, problem (\ref{fpu}).}
\begin{tabular}{|c|lllll|} \hline
$h$  & $1.6\cdot10^{-2}$ &$8\cdot10^{-3}$ & $4\cdot10^{-3}$ & $2\cdot10^{-3}$ &
$10^{-3}$ \\
\hline
error  & $3.030$ &$1.967\cdot10^{-1}$ &$1.240\cdot10^{-2}$ &$7.761\cdot10^{-4}$
&$4.853\cdot10^{-5}$ \\
\hline
order  & -- & 3.97& 3.99 &4.00 &4.00\\
\hline
\end{tabular}
\bigskip

\caption{\protect\label{tp3} Numerical order of convergence for
the HBVM(6,2) method, problem (\ref{biot}).}
\begin{tabular}{|c|lllll|} \hline
$h$  & $3.2\cdot10^{-2}$ &$1.6\cdot10^{-2}$ &$8\cdot10^{-3}$ & $4\cdot10^{-3}$ &
$2\cdot10^{-3}$ \\
\hline error  & $3.944\cdot10^{-6}$ &$2.635\cdot10^{-7}$
&$1.729\cdot10^{-8}$
&$1.094\cdot10^{-9}$ &$6.838\cdot10^{-11}$ \\
\hline
order  & -- & 3.90& 3.93 &3.98 &4.00\\
\hline
\end{tabular}
\bigskip

\caption{\protect\label{tab1} Maximum difference between the
numerical solutions obtained through the fourth-order HBVM$(k,2)$
methods based on Lobatto abscissae and Gauss-Legendre abscissae
for increasing values of $k$, problem (\ref{biot}), $10^3$ steps
with stepsize $h=0.1$.}\medskip \centerline{\begin{tabular}{|r|l|}
\hline $k$  &  $h=0.1$ \\
\hline
2  & $3.97 \cdot 10^{-1}$   \\
4  & $2.29 \cdot 10^{-3}$\\
6  & $2.01 \cdot 10^{-8}$\\
8  & $1.37 \cdot 10^{-11}$\\
10 & $5.88 \cdot 10^{-13}$\\
\hline
\end{tabular}}
\end{table}

\newpage

\section*{Test problem 4 (Sitnikov problem)}

The main problem in Celestial Mechanics is the so called $N$-body
problem, i.e. to describe the motion of $N$ point particles of
positive mass moving under Newton's law of gravitation when we
know their positions and velocities at a given time. This problem
is described by the Hamiltonian function:
\begin{equation}
\label{kepler} H(\bfq,\bfp)=\frac{1}{2} \sum_{i=1}^N
\frac{||p_i||_2^2}{m_i} - G \sum_{i=1}^N m_i
\sum_{j=1}^{i-1}\frac{m_j}{||q_i-q_j||_2},
\end{equation}

\no where $q_i$ is the position of the $i$th particle, with mass
$m_i$, and $p_i$ is its momentum.

The Sitnikov problem is a particular configuration of the $3$-body
dynamics (see, e.g., \cite{James}). In this problem two bodies of
equal mass (primaries) revolve about their center of mass, here
assumed at the origin, in elliptic orbits in the $xy$-plane. A
third, and much smaller body (planetoid), is placed on the
$z$-axis with initial velocity parallel to this axis as well.

The third body is small enough that the two body dynamics of the
primaries is not destroyed. Then, the motion of the third body will
be restricted to the $z$-axis and oscillating around the origin
but not necessarily periodic. In fact this problem has been shown
to exhibit a chaotic behavior  when the eccentricity of the orbits
of the primaries exceeds a critical value that, for the data set
we have used, is $\bar e \simeq 0.725$ (see Figure
\ref{sit_fig1}).

\begin{figure}[ht]
\begin{center}
\includegraphics[width=.9\textwidth,height=8cm]{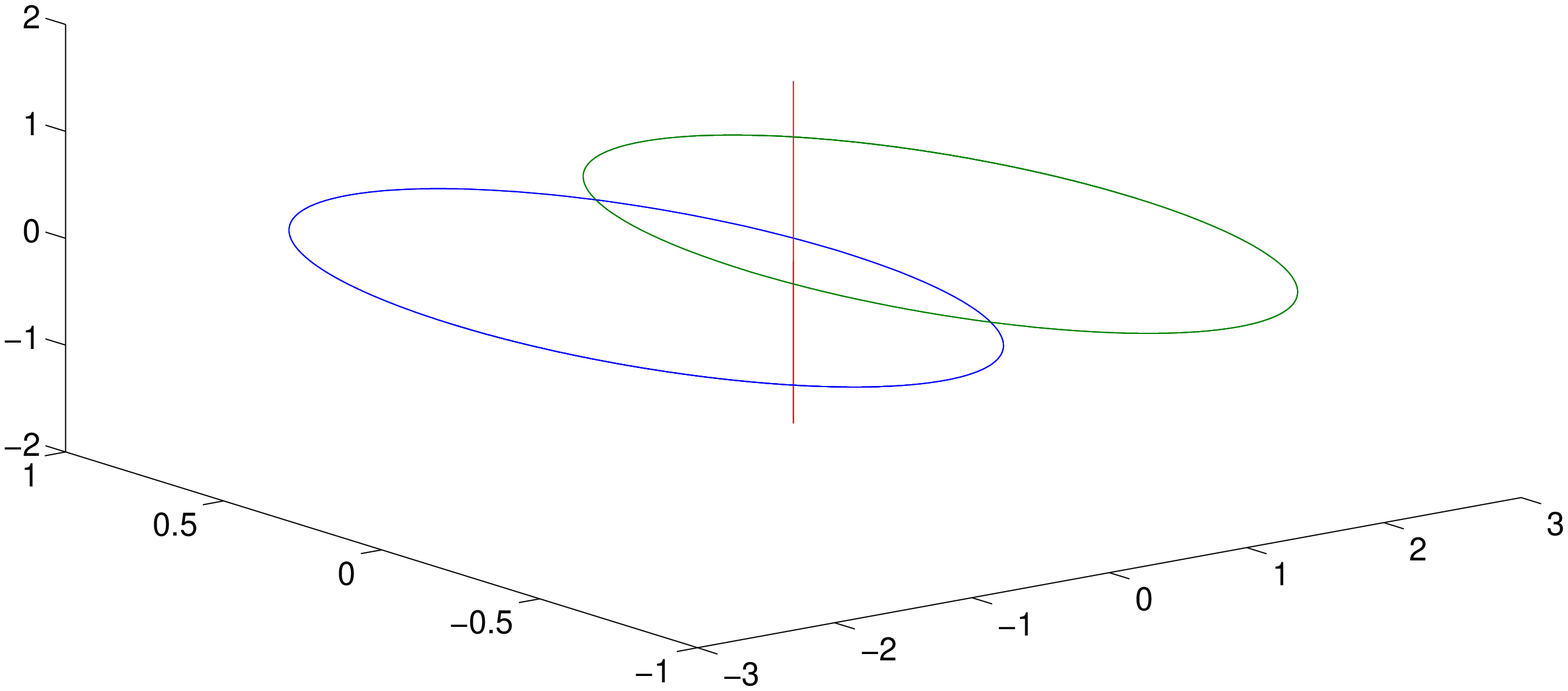}\\

\bigskip
\bigskip
\includegraphics[width=.9\textwidth,height=8cm]{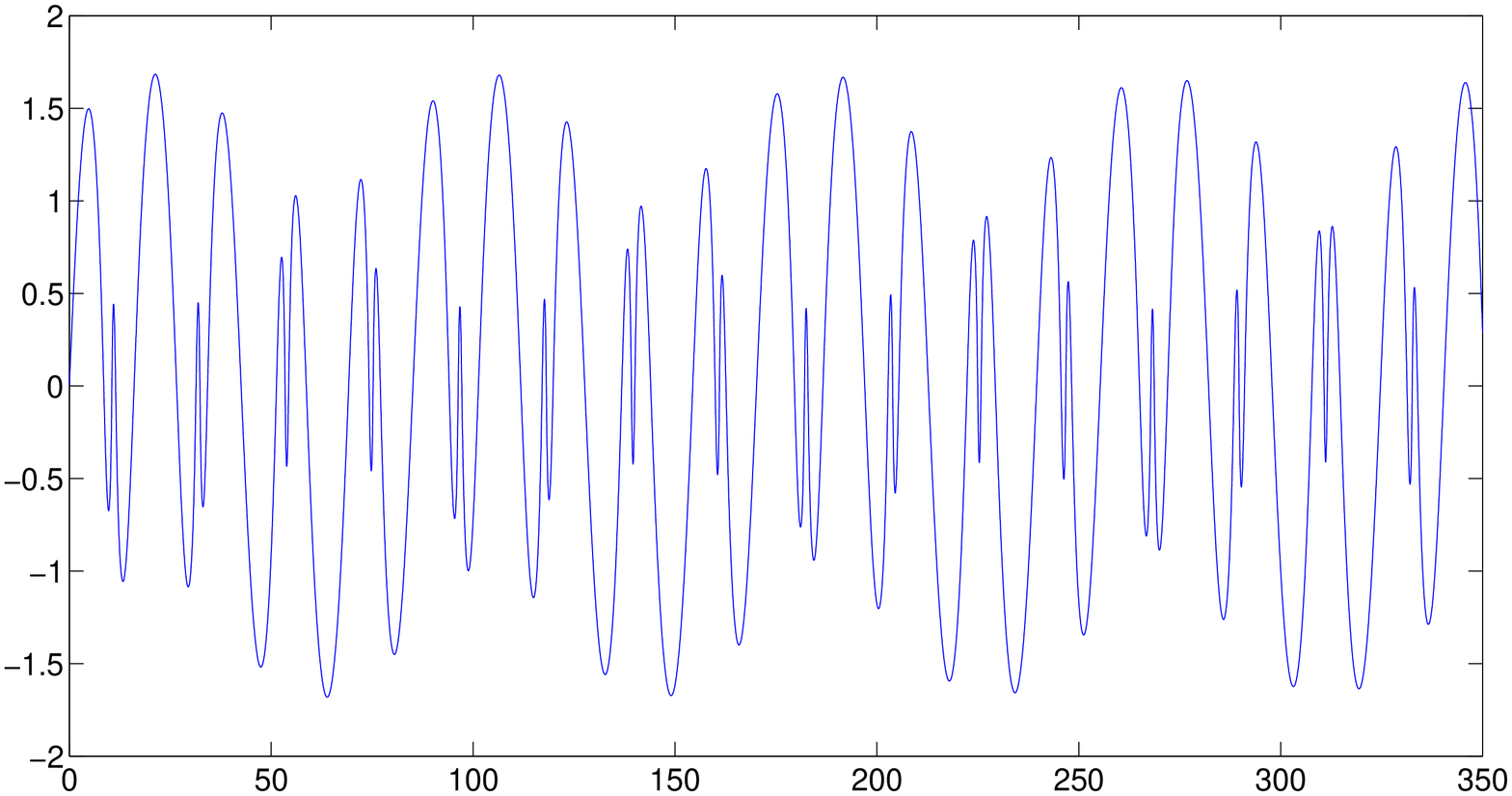}
\end{center}  \caption{The upper picture displays the configuration of
$3$-bodies in the Sitnikov problem. To an eccentricity of the
orbits of the primaries $e=0.75$, there correspond bounded chaotic
oscillations of the planetoid as is argued by looking at the
space-time diagram in the down picture.} \label{sit_fig1}
\end{figure}

We have solved the problem defined by the Hamiltonian function
\eqref{kepler} by the Gauss method of order 4 (i.e., HBVM(2,2) at
2 Gaussian nodes) and by HBVM(18,2) at 18 Gaussian nodes (order 4,
$2$ fundamental and $16$ silent stages), with the following set of
parameters in (\ref{kepler}):

\begin{center}
\begin{tabular}{ccccccccc}
$N$ & $G$ & $m_1$ & $m_2$ & $m_3$ & $e$ & $d$ & $h$ &
$t_{\mbox{max}}$
\\[.1cm]
\hline \\[-.4cm] $3$ & $1$ & $1$ & $1$ & $10^{-5}$ & $0.75$ & $5$ & $0.5$ &
$1500$
\end{tabular}
\end{center}

\no where $e$ is the eccentricity, $d$ is the distance of the
apocentres of the primaries (points at which the two bodies are
the furthest), $h$ is the used time-step, and
$[0,\,t_{\mbox{max}}]$ is the time integration interval. The
eccentricity $e$ and the distance $d$ may be used to define the
initial condition $[\bfq_0,\bfp_0]$ (see \cite{James} for the
details):
$$
\begin{array}{l}
\bfq_0 = [-\frac{5}{2},~  0,~  0,~  \frac{5}{2},~  0,~  0,~   0,~
0,~10^{-9}]^T,\\[.1cm]
\bfp_0 =[0,~ -\frac{1}{20}\sqrt{10},~  0,~    0,~
\frac{1}{20}\sqrt{10},~ 0,~ 0,~ 0,~ \frac{1}{2}]^T.
\end{array}
$$

First of all, we consider the two pictures in Figure
\ref{sit_fig5} reporting the relative errors in the  Hamiltonian
function and in the angular momentum evaluated along the numerical
solutions computed by the two methods. We know  that the
HBVM(18,2) precisely conserves Hamiltonian polynomial functions of
degree at most $18$. This accuracy is high enough to guarantee
that the nonlinear Hamiltonian function \eqref{kepler} is as well
conserved up to the machine precision (see the upper picture):
from a geometrical point of view this means that a local
approximation of the level curves of \eqref{kepler} by a
polynomial of degree $18$ leads to a negligible error. The Gauss
method exhibits a certain error in the Hamiltonian function while,
being this formula symplectic, it precisely conserves the angular
momentum, as is confirmed by looking at the down picture of Figure
\ref{sit_fig5}. The error in the numerical angular momentum
associated with the HBVM(18,2) undergoes some bounded
periodic-like oscillations.

\begin{figure}[ht]
\begin{center}
\includegraphics[width=.9\textwidth,height=8cm]{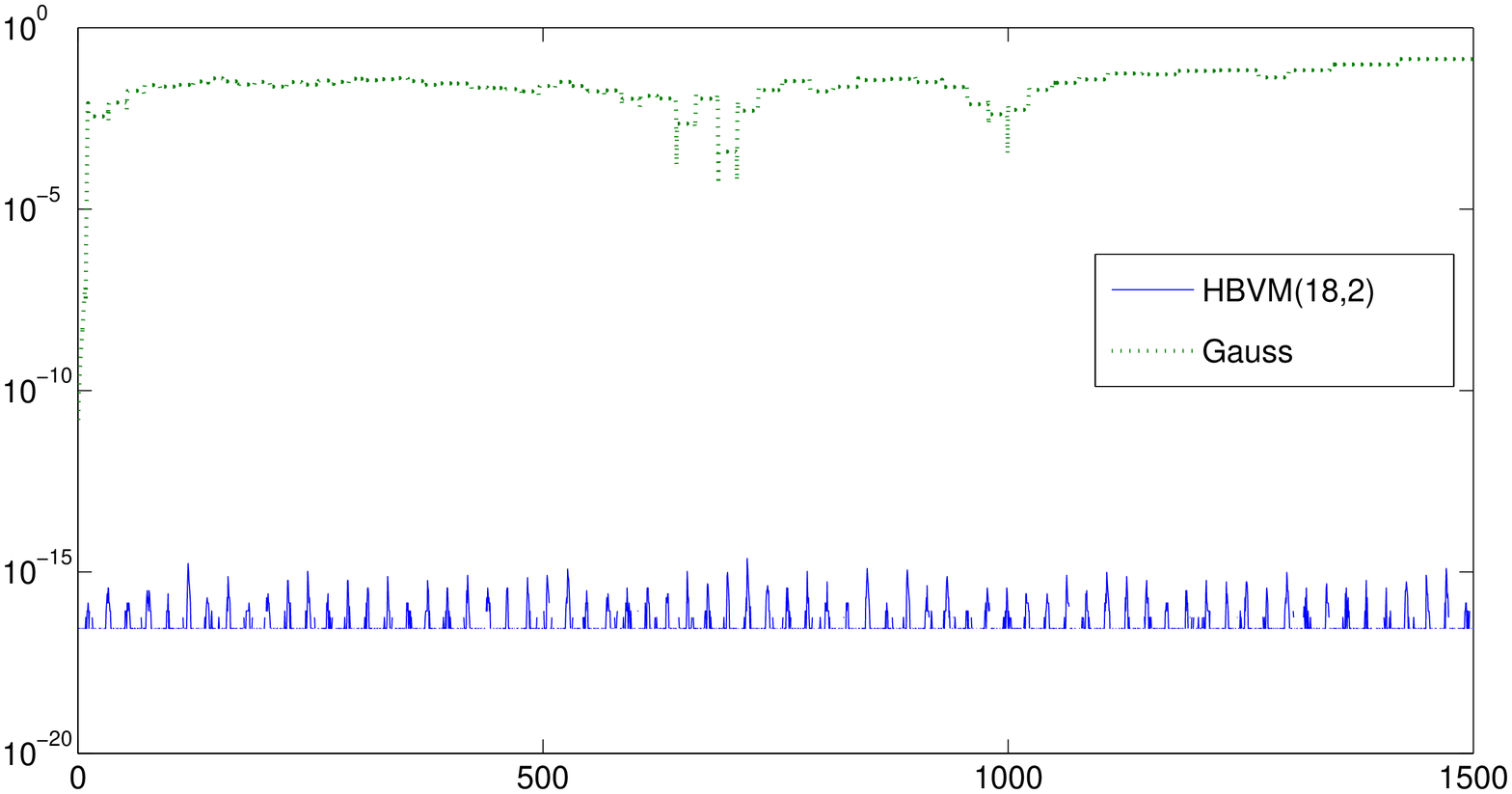}
\\

\bigskip\bigskip
\includegraphics[width=.9\textwidth,height=8cm]{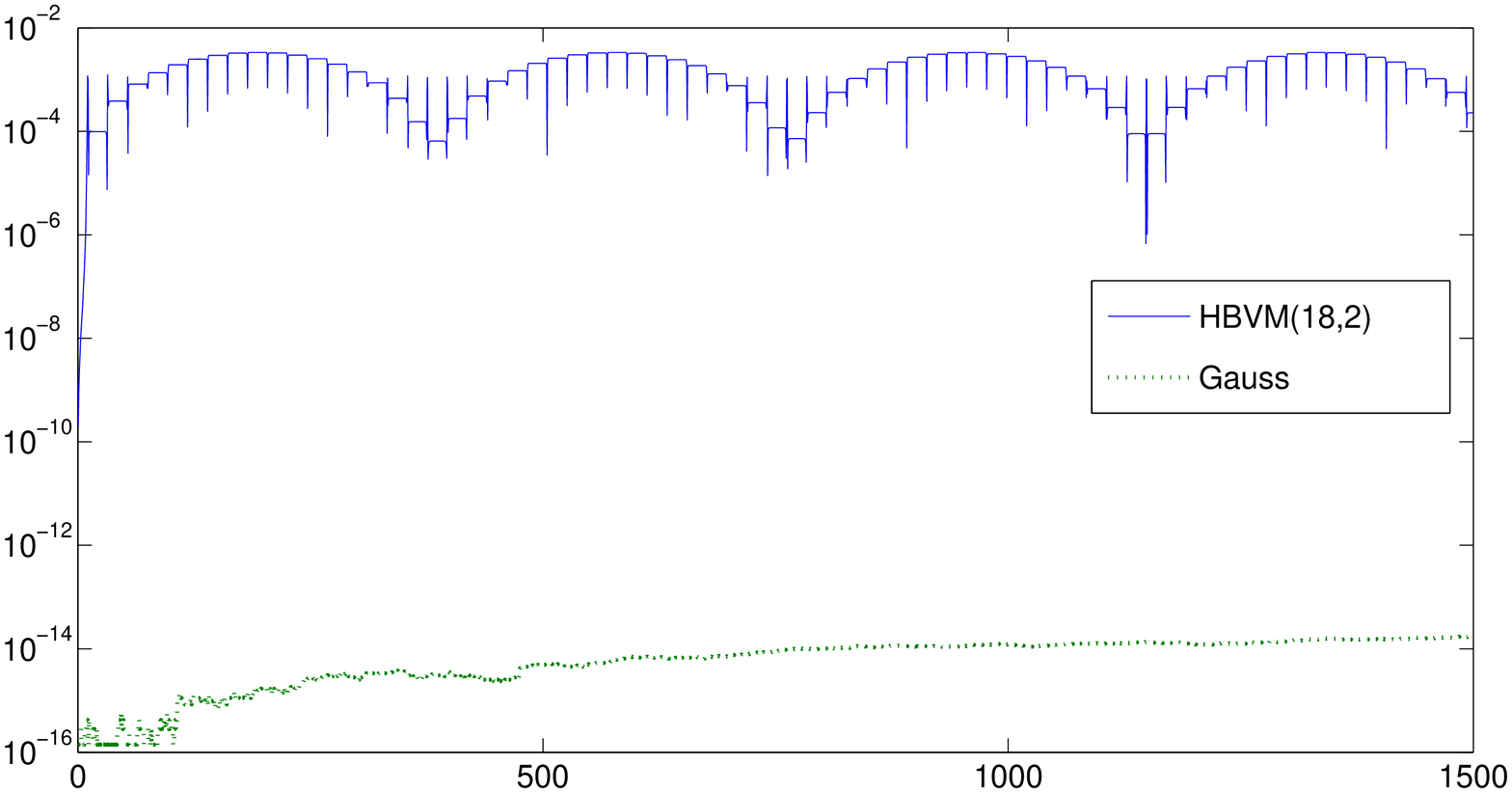}
\end{center} \caption{Upper picture: relative error $|H(y_n)-H(y_0)|/|H(y_0)|$
of the Hamiltonian function evaluated along the numerical solution
of the HBVM($18$,$2$) and the Gauss method. Down picture: relative
error $|M(y_n)-M(y_0)|/|M(y_0)|$ of the angular momentum evaluated
along the numerical solution of the HBVM($18$,$2$) and the Gauss
method.} \label{sit_fig5}
\end{figure}

Figures \ref{sit_fig2} and \ref{sit_fig3} show the numerical
solution computed by the Gauss method and HBVM(18,2), respectively.
Since the methods leave the $xy$-plane invariant for the motion of
the primaries and the $z$-axis invariant for the motion of the
planetoid, we have just reported the motion of the primaries in
the $xy$-phase plane (upper pictures) and the space-time diagram
of the planetoid (down picture).

\begin{figure}[ht]
\begin{center}
\includegraphics[width=.9\textwidth,height=8cm]{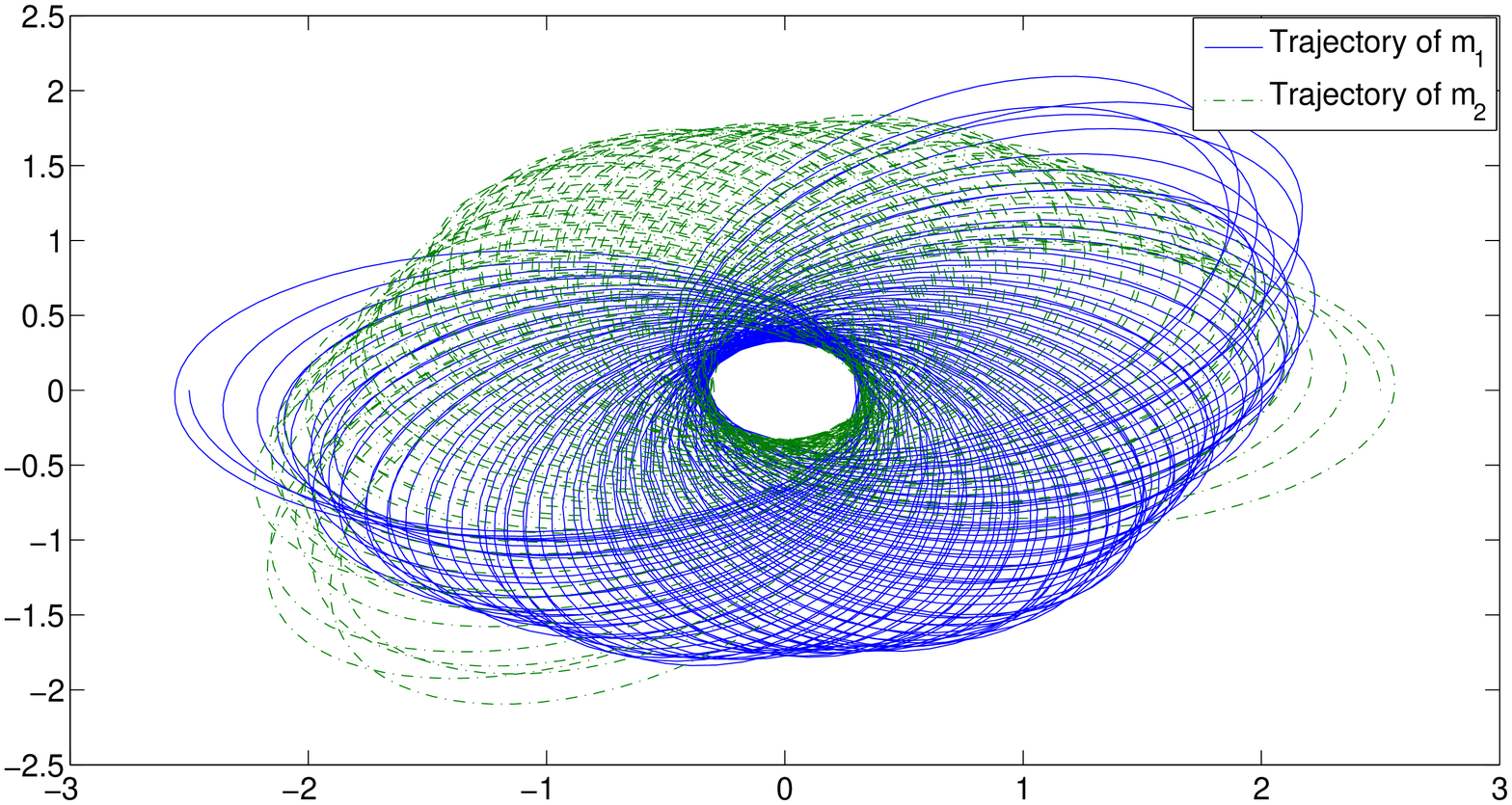}
\\

\bigskip\bigskip
\includegraphics[width=.9\textwidth,height=8cm]{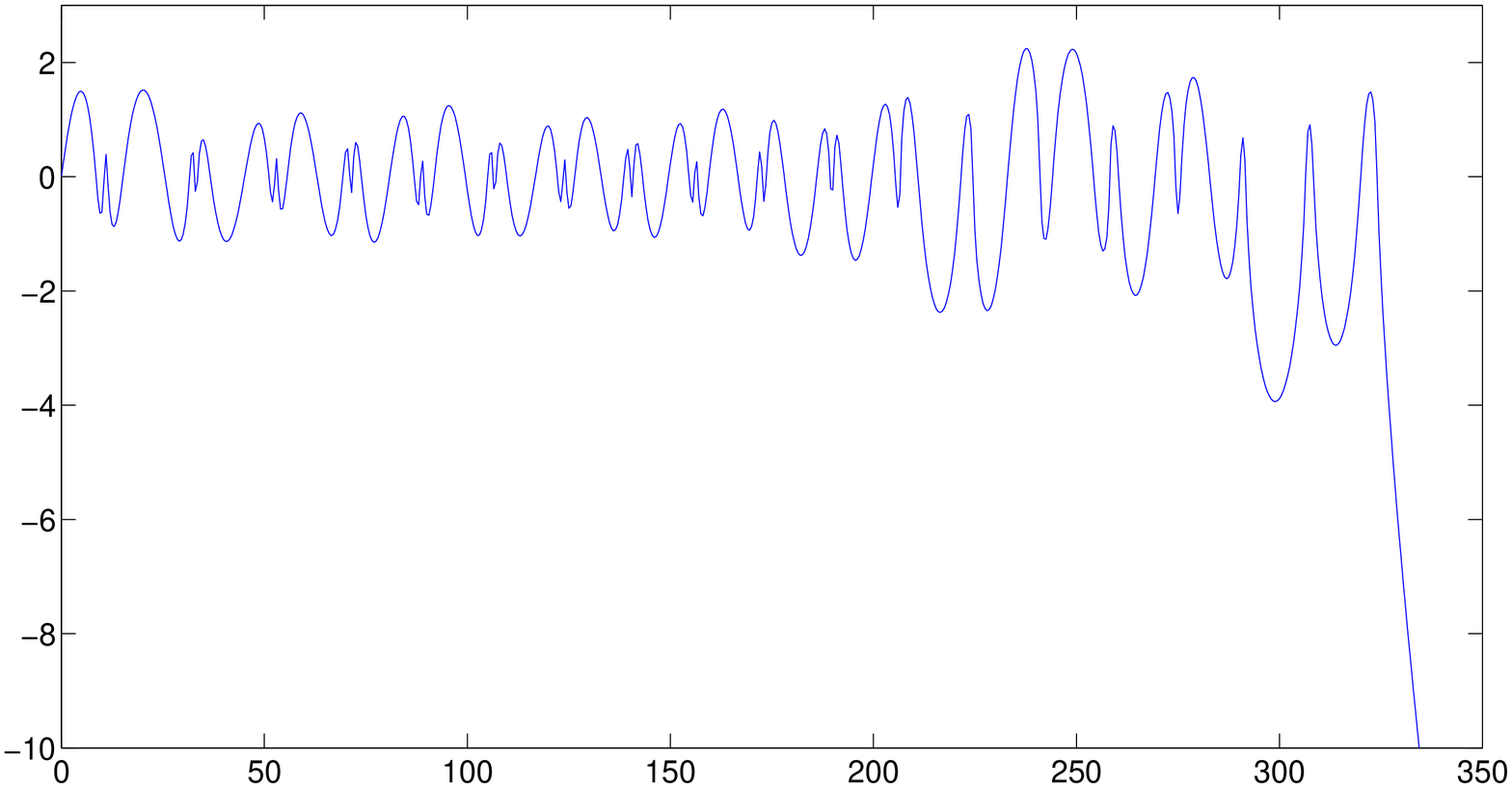}
\end{center}
\caption{The Sitnikov problem solved by the Gauss method of order
4, with stepsize $h=0.5$, in the time interval $[0,1500]$. The
trajectories of the primaries in the $xy$-plane (upper picture)
exhibit a very irregular behavior which causes the planetoid to
eventually escape the system, as illustrated by the space-time
diagram in the down picture.} \label{sit_fig2}
\end{figure}

We observe that, for the Gauss method, the orbits of the primaries
are irregular in character so that the third body, after
performing some oscillations around the origin, will eventually
escape the system (see the down picture of Figure
\ref{sit_fig2}). On the contrary (see the upper picture of Figure
\ref{sit_fig3}), the HBVM(18,2) method generates a quite regular phase
portrait. Due to the large stepsize $h$ used, a sham rotation of
the $xy$-plane appears which, however, does not destroy the global
symmetry of the dynamics, as testified by the bounded oscillations
of the planetoid (down picture of Figure \ref{sit_fig3}) which
look very similar to the reference ones in Figure \ref{sit_fig1}.
This aspect is also confirmed by the pictures in Figure
\ref{sit_fig4} displaying the distance of the primaries as a
function of the time. We see that the distance of the apocentres
(corresponding to the maxima in the plots), as the two bodies
wheel around the origin, are preserved by the HBVM(18,2) (down
picture) while the same is not true for the Gauss method (upper
picture).

\begin{figure}[ht]
\begin{center}
\includegraphics[width=.9\textwidth,height=8cm]{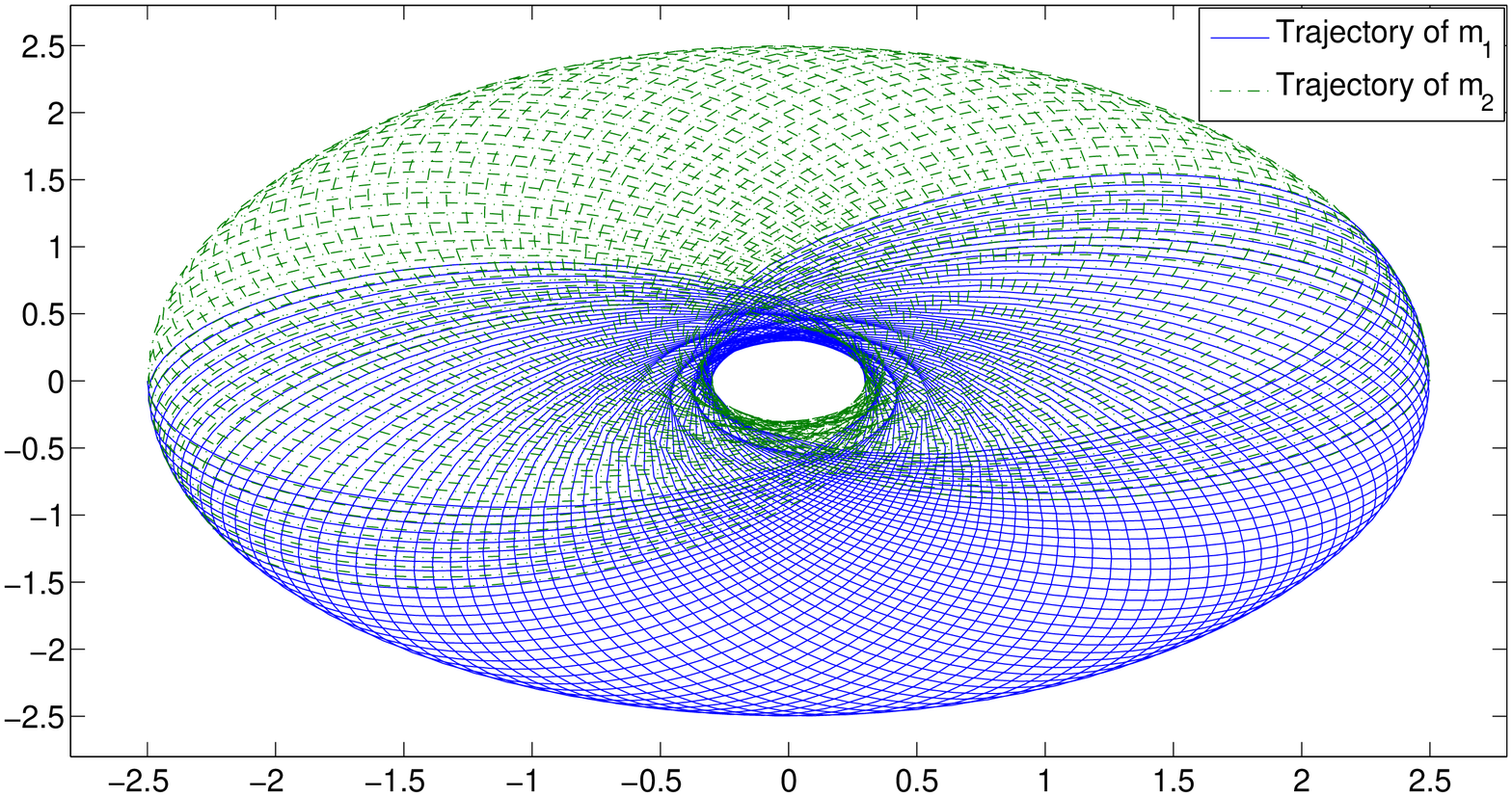}
\\

\bigskip\bigskip
\includegraphics[width=.9\textwidth,height=8cm]{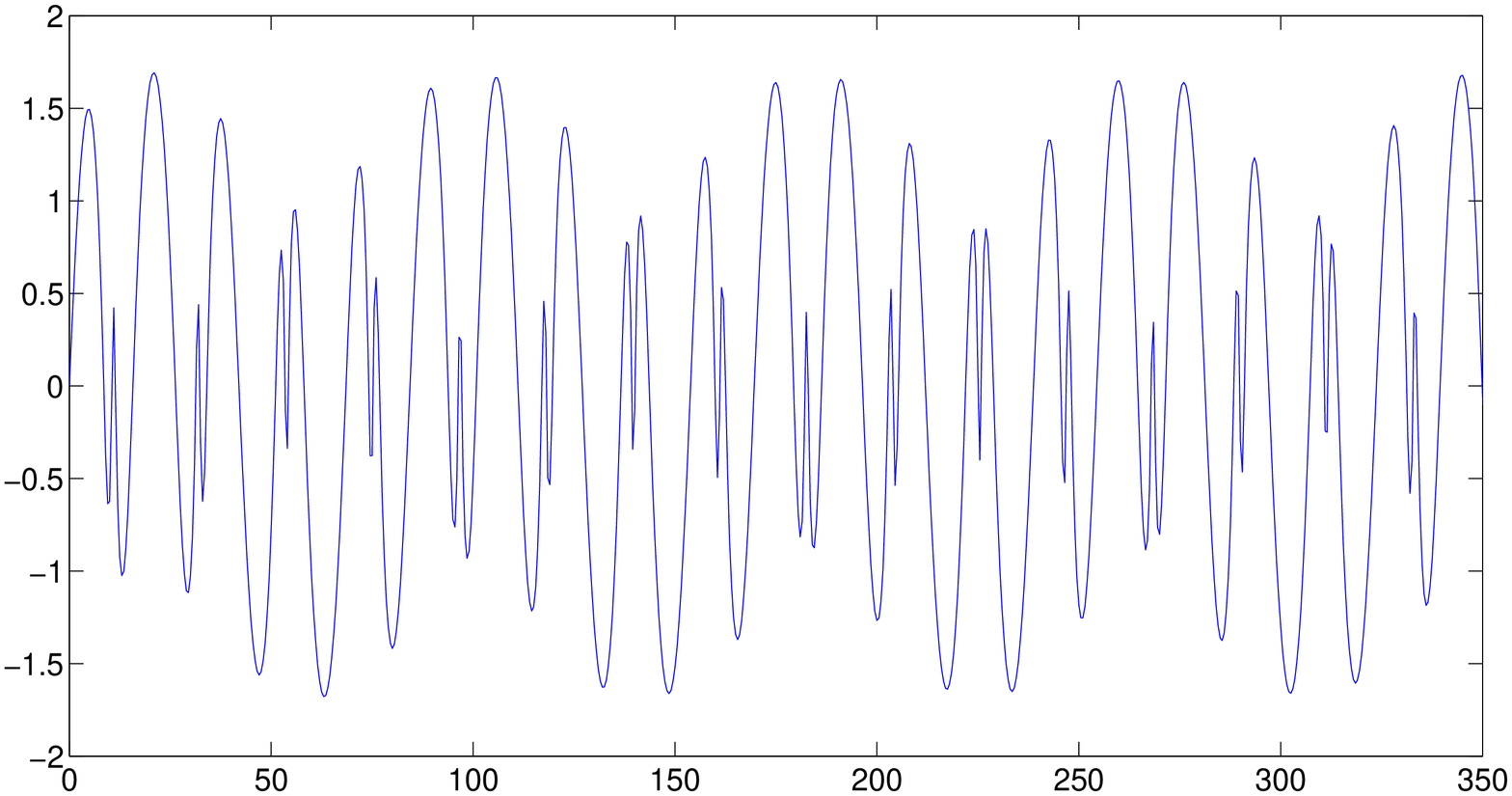}
\end{center}
\caption{The Sitnikov problem solved by the HBVM(18,2) method
(order 4), with stepsize $h=0.5$, in the time interval $[0,1500]$.
Upper picture: the trajectories of the primaries are ellipse
shape. The discretization introduces a fictitious uniform rotation
of the $xy$-plane which however does not alter the global symmetry
of the system. Down picture: the space-time diagram of the
planetoid on the $z$-axis displayed (for clearness) on the time
interval $[0, 350]$ shows that, although a large value of the
stepsize $h$ has been used, the overall behavior of the dynamics
is well reproduced (compare with the down picture in Figure
\ref{sit_fig1}).} \label{sit_fig3}
\end{figure}

\begin{figure}[ht]
\begin{center}
\includegraphics[width=.9\textwidth,height=8cm]{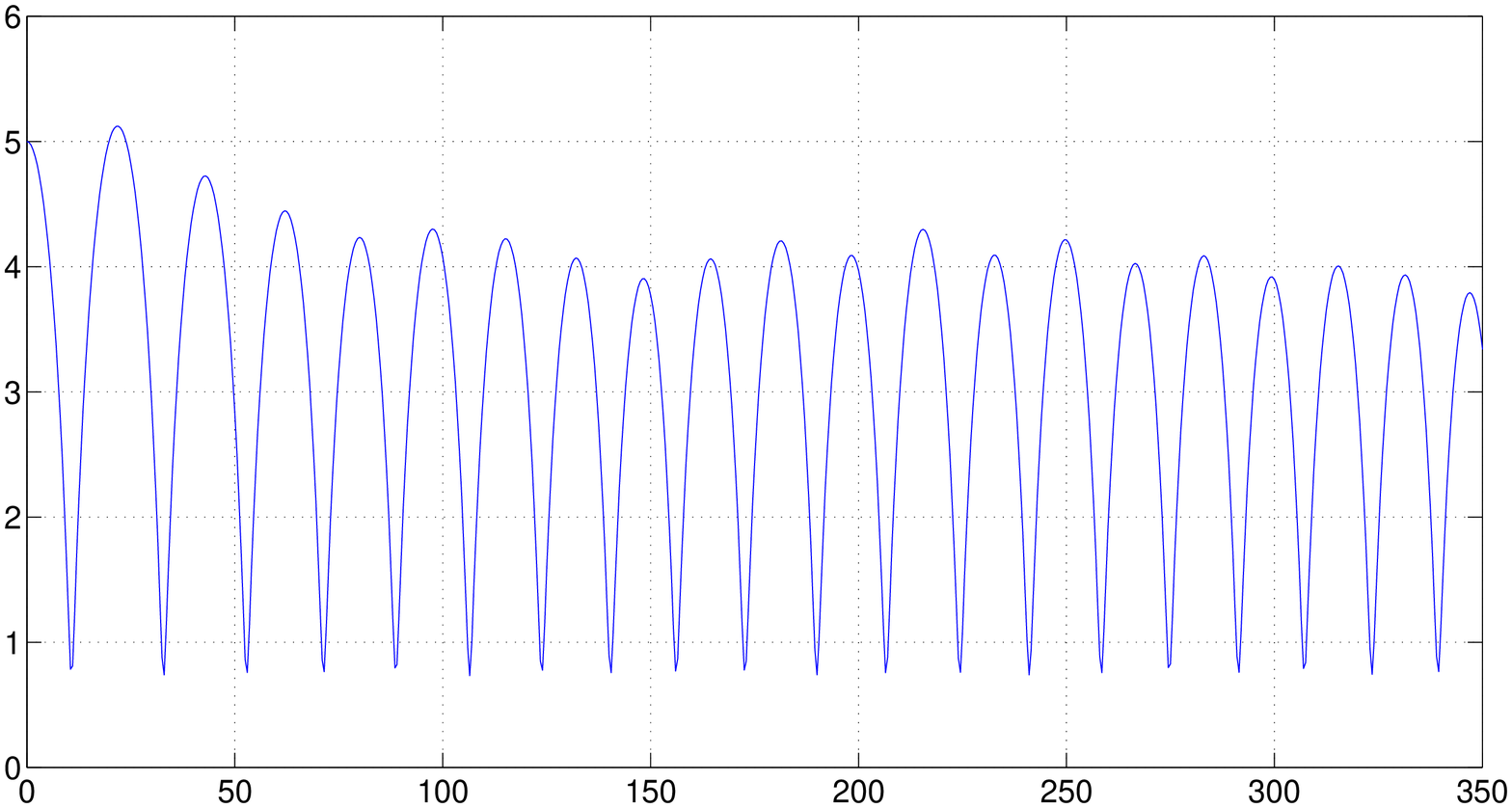}
\\

\bigskip\bigskip
\includegraphics[width=.9\textwidth,height=8cm]{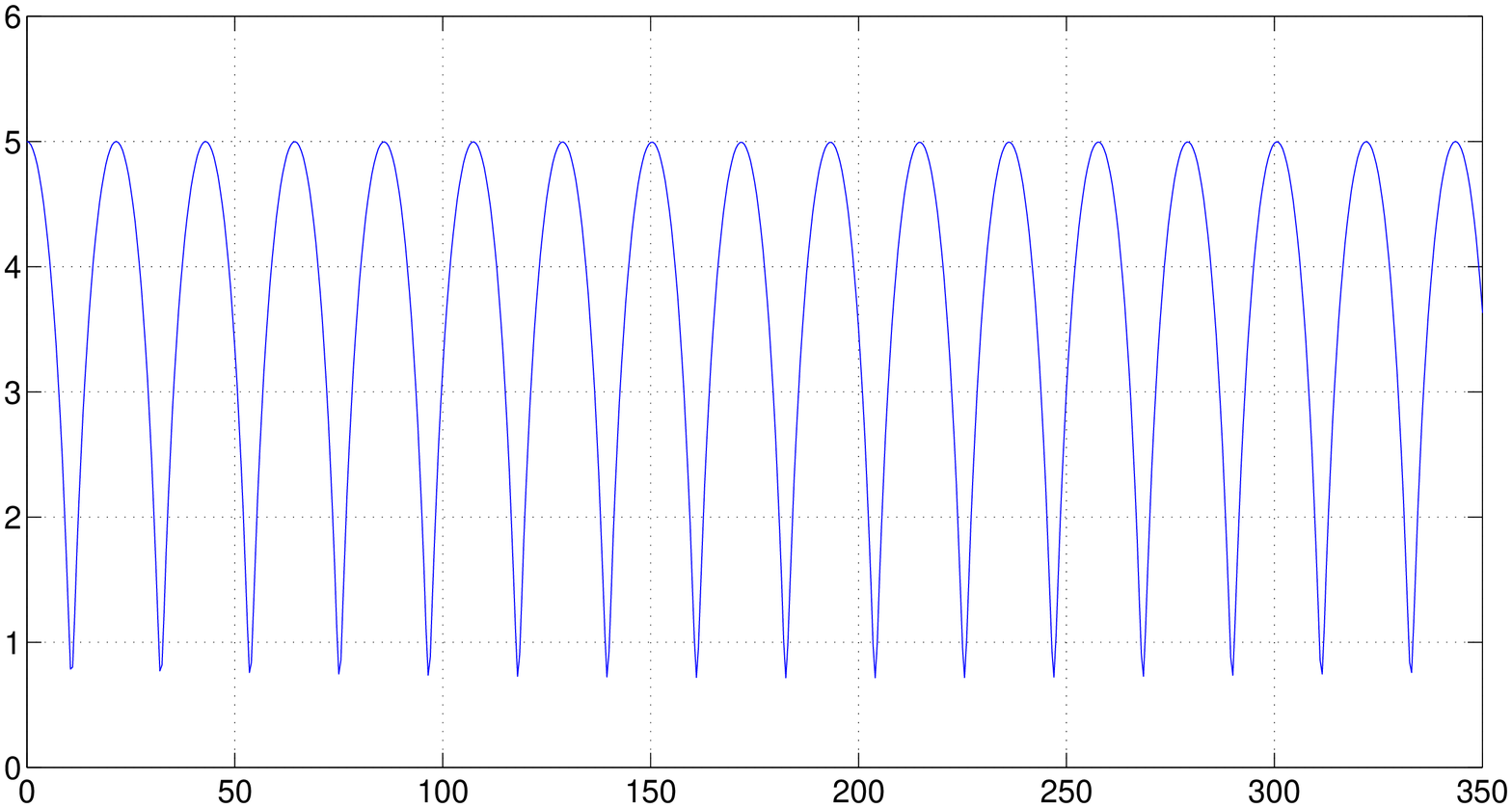}
\end{center}
\caption{Distance between the two primaries as a function of the
time, related to the numerical solutions generated by the Gauss
method (upper picture) and HBVM(18,2) (down picture). The maxima
correspond to the distance of apocentres. These are conserved by
HBVM(18,2) while the Gauss method introduces patchy oscillations
that destroy the overall symmetry of the system.} \label{sit_fig4}
\end{figure}

\chapter{Infinity HBVMs}\label{chap3}
 From the previous arguments, it is clear that the orthogonality
conditions (\ref{orth}), i.e., the fulfillment of the {\em Master
Functional Equation} (\ref{L}), is in principle only a sufficient
condition for the conservation property (\ref{conservation}) to
hold, when a generic polynomial basis $\{P_j\}$ is considered. Such
a condition becomes also necessary, when such basis is orthonormal.

\begin{theo}\label{ortbas} Let $\{P_j\}$ be an orthonormal basis on the
interval $[0,1]$. Then, assuming $H(y)$ to be analytical,
(\ref{conservation}) implies that each term in the sum has to
vanish.
\end{theo}

\begin{proof}
Let us consider the expansion$$g(\tau) \equiv \nabla H(\sigma(t_0+\tau h)) =
\sum_{\ell\ge 1} \rho_\ell P_\ell(\tau), \qquad \rho_\ell
=(P_\ell,g), \qquad \ell\ge1,$$

\no where, in general,
$$(f,g) = \int_0^1f(\tau)g(\tau) \dd\tau.$$

\no Substituting into (\ref{conservation}), yields
$$\sum_{j=1}^{s} \gamma_j^T (P_j,g)=\sum_{j=1}^{s} \gamma_j^T
\left(P_j,\sum_{\ell\ge 1} \rho_\ell P_\ell\right)=\sum_{j=1}^s
\gamma_j^T\rho_j=0.$$

\no Since this has to hold whatever the choice of the function
$H(y)$, one concludes that
\begin{equation}\label{ortoj}\gamma_j^T\rho_j=0, \qquad j=1,\dots,s.\QED
\end{equation}
\end{proof}\medskip

\begin{rem}\label{J}
In the case where $\{P_j\}$ is an orthonormal basis, from
(\ref{ortoj}) one then derives that $$\gamma_j = S \rho_j, \qquad
i=1,\dots,s,$$ where $S$ is any nonsingular skew-symmetric matrix.
The natural choice $S=J$ then leads to (\ref{orth}).
\end{rem}

Moreover, we observe that, if the Hamiltonian $H(y)$ is a
polynomial, the integral appearing at the right-hand side in
(\ref{hbvm_int}) is exactly computed by a quadrature formula, thus
resulting into a HBVM($k$,$s$) method with a sufficient number of
silent stages. As already stressed in the Chapter~\ref{chap1}, in
the non-polynomial case such formulae represent the limit of the
sequence HBVM($k$,$s$), as $k \rightarrow \infty$.

\begin{defn} For general Hamiltonians, we call the limit formula
(\ref{hbvm_int}) {\em Infinity Hamiltonian Boundary Value Method
of degree $s$} (in short, {\em $\infty$-HBVM of degree $s$} or
{\em HBVM$(\infty,s)$}) ~\cite{BIT09}.
\end{defn}

More precisely, due to the choice of the orthonormal basis
(\ref{orto1}),
$$ \mbox{HBVM}(\infty,s) = \lim_{k\rightarrow\infty}
\mbox{HBVM}(k,s),$$ whatever is the choice of the fundamental
abscissae $\{c_i\}$.

A worthwhile consequence of Theorems~\ref{ordine} and \ref{stab}
is that one can transfer to HBVM$(\infty,s)$  all those properties
of HBVM($k$,$s$) which are satisfied starting from a given $k \ge
k_0$ on: for example, the order and stability properties.

\begin{cor}
\label{ordineinf} Whatever the choice of the abscissae $c_1,\dots,c_s$,
HBVM$(\infty,s)$ \eqref{hbvm_int} has order $2s$ and is perfectly $A$-stable.
\end{cor}

\chapter[HBVMs and collocation methods]
{Isospectral Property of HBVMs and their connections with
Runge-Kutta collocation methods}\label{chap4}

When applied to initial value problems, HBVMs may be viewed as a
special subclass of Runge-Kutta (RK) methods of collocation type.
In Chapter \ref{chap1} (see also \cite{brugnano09bit,BIT09}) the
RK formulation turned out useful in  stating results pertaining to
the order of the new formulae. Here, the RK notation will be
exploited to derive the isospectral property of HBVMs and
elucidate the existing connections between HBVMs and RK
collocation methods \cite{BIT10_1}. In doing this, our aim is
twofold:

\begin{enumerate}
\item to better elucidate the close link
between the new formulae and the classical collocation Runge-Kutta
methods;

\item to make the handling of the new formulae more comfortable to
the scientific community working in the context of RK methods.
\end{enumerate}

In fact, we think that HBVMs (and consequently their RK
formulation) may be of interest beyond their application to
Hamiltonian systems. Each HBVM($k$,$s$) becomes a classical
collocation method when $k=s$, while, for $k>s$, it conserves all
the features of the generating collocation formula, including the
order (which may be even improved, reaching eventually order
$p=2s$) and the dimension of the associated nonlinear system.

Let us then consider the matrix appearing in the Butcher tableau
(\ref{rk}), corresponding to HBVM$(k,s)$, i.e., the matrix
\begin{equation}\label{AMAT}A = \I_s \P_s^T\O\in\RR^{k\times
k}, \qquad k\ge s,\end{equation}

\no whose rank is $s$ (see (\ref{OIP})--(\ref{IDPO})).
Consequently it has a $(k-s)$-fold zero eigenvalue. To begin with,
we are going to discuss the location of the remaining $s$
eigenvalues of that matrix.

Before that, we state the following preliminary result, whose
proof can be found in \cite[Theorem\,5.6 on page\,83]{HW}.

\begin{lem}\label{gauss} The eigenvalues of the matrix
\begin{equation}\label{Xs}
X_s = \pmatrix{cccc}
\frac{1}2 & -\xi_1 &&\\
\xi_1     &0      &\ddots&\\
          &\ddots &\ddots    &-\xi_{s-1}\\
          &       &\xi_{s-1} &0\\
\endpmatrix, \end{equation}

\no with
\begin{equation}\label{xij}\xi_j=\frac{1}{2\sqrt{(2j+1)(2j-1)}}, \qquad
j=1,\dots,s-1,\end{equation}

\no coincide with those of the matrix in the Butcher tableau of
the Gauss-Legendre method of order $2s$.\end{lem}

\medskip
We also need the following preliminary result, whose proof derives
from the properties of shifted-Legendre polynomials (see, e.g.,
\cite{AS} or the Appendix in \cite{brugnano09bit}).

\begin{lem}\label{intleg} With reference to the matrices in
(\ref{OIP})--(\ref{IDPO}), one has \begin{equation}\label{IPG}\I_s
= \P_{s+1}\hat{X}_s,\end{equation} where
\begin{equation}\label{G}
\hat{X}_s = \pmatrix{cccc}
\frac{1}2 & -\xi_1 &&\\
\xi_1     &0      &\ddots&\\
          &\ddots &\ddots    &-\xi_{s-1}\\
          &       &\xi_{s-1} &0\\
\hline &&&\xi_s\endpmatrix, \end{equation} with the $\xi_j$ defined by
(\ref{xij}). \end{lem}

\medskip
The following result then holds true  \cite{BIT10}.

\begin{theo}[\bf Isospectral Property of HBVMs]\label{mainres}
For all $k\ge s$ and for any choice of the abscissae $\{\tau_i\}$
such that $B(2s)$ holds true, the nonzero eigenvalues of the
matrix $A$ in (\ref{AMAT}) coincide with those of the matrix of
the Gauss-Legendre method of order $2s$.
\end{theo}

\begin{proof}
For $k=s$, the abscissae $\{\tau_i\}$ have to be the $s$
Gauss-Legendre nodes on $[0,1]$, so that HBVM$(s,s)$ reduces to
the Gauss Legendre method of order $2s$, as already observed in
Example~\ref{gaussex}.

When $k>s$, from the orthonormality of the basis, see
(\ref{orto}), and considering that the quadrature with weights
$\{\omega_i\}$ is exact for polynomials of degree (at least)
$2s-1$, one easily obtains that
$$\P_s^T\O\P_{s+1} = \left( I_s ~ \bfo\right),$$

\no since, for all $i=1,\dots,s$, ~and~ $j=1,\dots,s+1$:
$$\left(\P_s^T\O\P_{s+1}\right)_{ij} = \sum_{\ell=1}^k \omega_\ell
P_i(\tau_\ell)P_j(\tau_\ell) =\int_0^1 P_i(t)P_j(t)\dd
t=\delta_{ij}.$$

\no By taking into account the result of Lemma~\ref{intleg}, one
then obtains:
\begin{eqnarray}\nonumber
A\P_{s+1} &=& \I_s \P_s^T\O\P_{s+1} = \I_s \left(I_s~\bfo\right) =\P_{s+1}
\hat{X}_s \left(I_s~\bfo\right) = \P_{s+1}\left(\hat{X}_s~\bfo\right)\\
 &=& \P_{s+1}
\pmatrix{cccc|c}
\frac{1}2 & -\xi_1 && &0\\
\xi_1     &0      &\ddots& &\vdots\\
          &\ddots &\ddots    &-\xi_{s-1}&\vdots\\
          &       &\xi_{s-1} &0&0\\
\hline &&&\xi_s&0\endpmatrix \equiv \P_{s+1} \widetilde
X_s,\label{tXs}
\end{eqnarray}

\no with the $\{\xi_j\}$ defined according to (\ref{xij}).
Consequently, one obtains that the columns of $\P_{s+1}$
constitute a basis of an invariant (right) subspace of matrix $A$,
so that the eigenvalues of $\widetilde X_s$ are eigenvalues of
$A$. In more detail, the eigenvalues of $\widetilde X_s$ are those
of $X_s$ (see (\ref{Xs})) and the zero eigenvalue. Then, also in
this case, the nonzero eigenvalues of $A$ coincide with those of
$X_s$, i.e., with the eigenvalues of the matrix defining the
Gauss-Legendre method of order $2s$.\QED
\end{proof}

\section{HBVMs and collocation methods}

By using the previous result and notations, now we go to elucidate
the existing connections between HBVMs and RK collocation methods.
We shall continue to use an orthonormal basis $\{P_j\}$, along
which the underlying {\em extended collocation} polynomial
$\sigma(t)$ is expanded, even though the arguments could be
generalized to more general bases, as sketched below. On the other
hand, the distribution of the internal abscissae can be arbitrary.

Our starting point is a generic collocation method with $k$
stages, defined by the tableau
\begin{equation}
\label{collocation_rk}
\begin{array}{c|c}\begin{array}{c} \tau_1\\ \vdots\\ \tau_k\end{array} &  \mathcal A \\
 \hline                    &\omega_1\, \ldots  ~ \omega_k
\end{array}
\end{equation}
where, for $i,j=1,\dots,k$, $\mathcal A=
\left(\alpha_{ij}\right)\equiv\left(\int_0^{\tau_i} \ell_j(\tau)
\mathrm{d}\tau \right)$ and $\omega_j=\int_0^{1} \ell_j(\tau)
\mathrm{d}\tau$, $\ell_j(t)$ being the $j$th Lagrange polynomial
of degree $k-1$ defined on the set of abscissae $\{\tau_i\}$.

Given a positive integer $s\le k$, we can consider a basis
$\{p_1(\tau), \dots, p_s(\tau)\}$ of the vector space of
polynomials of degree at most $s-1$, and we set
\begin{equation}
\label{P} \hat\P_s = \pmatrix{cccc}
p_1(\tau_1) & p_2(\tau_1) & \cdots & p_s(\tau_1) \\
p_1(\tau_2) & p_2(\tau_2) & \cdots & p_s(\tau_2) \\
\vdots   & \vdots   &        & \vdots \\
p_1(\tau_k) & p_2(\tau_k) & \cdots & p_s(\tau_k)
\endpmatrix_{k \times s}
\end{equation}

\no (note that $\hat\P_s$ is full rank since the nodes are
distinct). The class of RK methods we are interested in is defined
by the tableau
\begin{equation}
\label{hbvm_rk}
\begin{array}{c|c}\begin{array}{c} \tau_1\\ \vdots\\ \tau_k\end{array} &
A \equiv \mathcal A \hat\P_s \Lambda_s \hat\P_s^T \Omega\\
 \hline                    &\omega_1\, \ldots \ldots ~ \omega_k
\end{array}
\end{equation}

\no where $\Omega=\diag(\omega_1,\dots,\omega_k)$ and
$\Lambda_s=\diag(\eta_1,\dots,\eta_s)$; the coefficients $\eta_j$,
$j=1,\dots,s$, have to be selected by imposing suitable
consistency conditions on the stages $\{Y_i\}$ \cite{BIT09}. In
particular, when the basis is orthonormal, as we shall assume
hereafter, then matrix $\hat\P_s$ reduces to matrix $\P_s$ in
(\ref{OIP})--(\ref{IDPO}), $\Lambda_s = I_s$, and consequently
(\ref{hbvm_rk}) becomes
\begin{equation}
\label{hbvm_rk1}
\begin{array}{c|c}\begin{array}{c} \tau_1\\ \vdots\\ \tau_k\end{array} &
A \equiv \mathcal A \P_s \P_s^T \Omega\\
 \hline       &\omega_1\, \ldots \ldots ~ \omega_k
\end{array}
\end{equation}

We note that the Butcher array $A$ has rank which cannot exceed
$s$, because it is defined by {\em filtering} $\mathcal A$ by the
rank $s$ matrix $\P_s \P_s^T \Omega$.

The following result then holds true, which clarifies the existing
connections between classical RK collocation methods and HBVMs.

\begin{theo}\label{collhbvm} Provided that the quadrature formula defined by the
weights $\{\omega_i\}$ is exact for polynomials at least $2s-1$
(i.e., the RK method defined by the tableau (\ref{hbvm_rk1})
satisfies the usual simplifying assumption $B(2s)$), then the
tableau (\ref{hbvm_rk1}) defines a HBVM$(k,s)$ method based at the
abscissae $\{\tau_i\}$.
\end{theo}

\proof Let us expand the basis $\{P_1(\tau),\dots,P_s(\tau)\}$
along the Lagrange basis $\{\ell_j(\tau)\}$, $j=1,\dots,k$,
defined over the nodes $\tau_i$, $i=1,\dots,k$: $$
P_j(\tau)=\sum_{r=1}^k P_j(\tau_r) \ell_r(\tau),
 \qquad j=1,\dots,s.$$

\no It follows that, for $i=1,\dots,k$ and $j=1,\dots,s$:
$$\int_0^{\tau_i} P_j(x) \mathrm{d}x = \sum_{r=1}^k P_j(\tau_r)
\int_0^{\tau_i} \ell_r(x) \mathrm{d}x = \sum_{r=1}^k P_j(\tau_r)
\alpha_{ir},$$

\no that is (see (\ref{OIP})--(\ref{IDPO}) and
(\ref{collocation_rk})),
\begin{equation}\label{APeqI}
\I_s = \mathcal A \P_s.
\end{equation}

\no By substituting (\ref{APeqI}) into (\ref{hbvm_rk1}), one
retrieves that tableau (\ref{rk}), which defines the method
HBVM$(k,s)$. This completes the proof.\QED

\medskip
The resulting Runge-Kutta method \eqref{hbvm_rk1} is then energy
conserving if applied to polynomial Hamiltonian systems
\eqref{hamilode} when the degree of $H(y)$, is lower than or equal
to a quantity, say $\nu$, depending on $k$ and $s$. As an example,
when a Gaussian distribution of the nodes $\{\tau_i\}$ is
considered, one obtains (\ref{knu}).

\begin{rem}[{\bf About Simplecticity}]\label{symplectic} The choice
of the abscissae $\{\tau_1,\dots,\tau_k\}$ at the Gaussian points
in $[0,1]$ has also another important consequence, since, in such
a case, the collocation method (\ref{collocation_rk}) is the Gauss
method of order $2k$ which, as is well known, is a {\em symplectic
method}. The result of Theorem~\ref{collhbvm} then states that,
for any $s\le k$, the HBVM$(k,s)$ method is related to the Gauss
method of order $2k$ by the relation: $$A = {\cal A}
(\P_s\P_s^T\O),$$

\no where the {\em filtering matrix} $(\P_s\P_s^T\O)$ essentially
makes the Gauss method of order $2k$ ``work'' in a suitable
subspace.
\end{rem}

It seems like the price paid to achieve such conservation
properties consists in the lowering of the order of the new method
with respect to the original one \eqref{collocation_rk}. Actually
this is not true,  because a fair comparison would be to relate
method \eqref{rk}--\eqref{hbvm_rk1}  to a collocation method
constructed on $s$ rather than on $k$  stages. This fact will be
fully elucidated in Chapter~\ref{chap5}.

\subsection{An alternative proof for the order of HBVMs}

We conclude this chapter by observing that the order $2s$ of an
HBVM$(k,s)$ method, under the hypothesis that
\eqref{collocation_rk} satisfies the usual simplifying assumption
$B(2s)$, i.e., the quadrature defined by the weights
$\{\omega_i\}$ is exact for polynomials of degree at least $2s-1$,
may be stated by using an alternative, though equivalent,
procedure to that used in the proof of Theorem~\ref{ordine}.

Let us then define the $k \times k$ matrix $\P\equiv \P_k$  (see
(\ref{OIP})--(\ref{IDPO})) obtained by ``enlarging'' the matrix
$\P_s$ with $k-s$ columns defined by the normalized shifted
Legendre polynomials $P_j(\tau)$, $j=s+1,\dots,k$, evaluated at
$\{\tau_i\}$, i.e.,
$$\P=\pmatrix{ccc} P_1(\tau_1) & \dots &P_k(\tau_1)\\ \vdots &
&\vdots\\ P_1(\tau_k) & \dots &P_k(\tau_k)\endpmatrix.$$

\no By virtue of property $B(2s)$ for the quadrature formula
defined by the weights $\{\omega_i\}$, it satisfies
$$
\P^T \O \P =\pmatrix{ll} I_s & O \\ O & R
\endpmatrix, \qquad R\in\RR^{k-s\times k-s}.
$$

\no  This implies that $\P$ satisfies the property $T(s,s)$ in
\cite[Definition\,5.10 on page 86]{HW}, for the quadrature formula
$(\omega_i,\tau_i)_{i=1}^k$. Therefore, for the matrix $A$
appearing in \eqref{hbvm_rk1} (i.e., (\ref{rk}), by virtue of
Theorem~\ref{collhbvm}), one obtains:
\begin{equation}
\label{rk_leg1} \P^{-1} A \P = \P^{-1} \mathcal A \P \pmatrix{ll} I_s \\
& O\endpmatrix = \pmatrix{ll} \widetilde{X}_s \\ & O
\endpmatrix,
\end{equation}

\no where $\widetilde X_s$ is the matrix defined in (\ref{tXs}).
Relation \eqref{rk_leg1} and \cite[Theorem\,5.11 on page 86]{HW}
prove that method (\ref{hbvm_rk1}) (i.e., HBVM$(k,s)$) satisfies
$C(s)$ and $D(s-1)$ and, hence, its order is $2s$.

\begin{rem}[{\bf Invariance of the order}] From the previous result
 we deduce the invariance of the superconvergence property
 of HBVM($k$,$s$) with respect to the distribution of the  abscissae
$\tau_i$, $i=1,\dots,k$, the only assumption to get the order $2s$
being that the underlying quadrature formula has degree of
precision $2s-1$. Such exceptional circumstance is likely to have
interesting applications beyond the purposes here presented.
\end{rem}

\chapter{Blended HBVMs}\label{chap5}

We shall now consider some computational aspects concerning
HBVM$(k,s)$. In more details, we now show how its cost depends
essentially on $s$, rather than on $k$, in the sense that the
nonlinear system to be solved, for obtaining the discrete
solution, has (block) dimension $s$
\cite{BIS,brugnano09bit,BIT10}.

This could be inferred from the fact that the silent stages
(\ref{hYi}) depend on the fundamental stages: let us see the
details. In order to simplify the notation, we shall fix the
fundamental stages at $\tau_1,\dots,\tau_s$, since we have already
seen that, due to the use of an orthonormal basis, they could be
in principle chosen arbitrarily, among the abscissae $\{\tau_i\}$.
With this premise, we have, from (\ref{discr_lin}),
(\ref{aij})--(\ref{hbvm_int}), and by using the notation
(\ref{tiyi}),
\begin{equation}\label{ys}
y_i = y_0 + h\sum_{j=1}^s a_{ij}  \sum_{\ell = 1}^k\omega_\ell
P_j(\tau_\ell)f_\ell, \qquad i = 1,\dots,s.\end{equation}

This equation is now coupled with that defining the silent stages,
i.e., from (\ref{expan}) and (\ref{hYi}),
\begin{equation}\label{hy}
y_i = y_0 + h\sum_{j=1}^s \cc_j \int_0^{\tau_i}P_j(t) \dd t,
\qquad i = s+1,\dots,k.
\end{equation}

Let us now partition the matrices $\I_s,\P_s\in\RR^{k\times s}$ in
(\ref{OIP})--(\ref{IDPO}) into $$\I_{s1},\P_{s1}\in\RR^{s\times
s},\qquad \I_{s2},\P_{s2}\in\RR^{k-s\times s},$$

\no containing the entries defined by the fundamental abscissae
and the silent abscissae, respectively. Similarly, we partition
the vector $\bfy$ into $\bfy_1$, containing the fundamental
stages, and $\bfy_2$ containing the silent stages and,
accordingly, let
$$\O_1\in\RR^{s\times s}, \qquad \O_2\in\RR^{k-s\times k-s},$$

\no be the diagonal matrices containing the corresponding entries
in matrix $\O$. Finally, let us define the vectors $$\bfgam =
(\cc_1,\dots,\cc_s)^T, \qquad e=(1,\dots,1)^T\in\RR^s, \qquad u =
(1,\dots,1)^T\in\RR^{k-s}.$$

\no Consequently, we can rewrite (\ref{ys}) and (\ref{hy}), as
\begin{eqnarray}\label{y1}
\bfy_1 &=& e\otimes y_0 + h\I_{s1} \left( \P_{s1}^T~
\P_{s2}^T\right) \pmatrix{cc}\O_1 \\ &\O_2\endpmatrix\otimes I_{2m}
\pmatrix{c} f(\bfy_1)\\
f(\bfy_2)\endpmatrix,\\ \bfy_2 &=& u\otimes y_0 +h \I_{s2}\otimes
I_{2m} \bfgam, \label{y2}\end{eqnarray}

\no respectively. The vector $\bfgam$ can be obtained by the
identity (see (\ref{y_i}))
$$\bfy_1 = e\otimes y_0 + h\I_{s1}\otimes I_{2m} \bfgam,$$

\no thus giving \begin{eqnarray}\nonumber \bfy_2 &=&
\left(u-\I_{s2}\I_{s1}^{-1}e\right)\otimes y_0
+\I_{s2}\I_{s1}^{-1}\otimes I_{2m} \bfy_1\\ &\equiv& \hat u\otimes
y_0 +A_1\otimes I_{2m}\bfy_1,\label{y2_1}\end{eqnarray}

\no in place of (\ref{y2}), where, evidently,
\begin{equation}\label{A1}\hat u = \left(u-\I_{s2}\I_{s1}^{-1}e\right)\in
\RR^{k-s}, \qquad A_1=\I_{s2}\I_{s1}^{-1}\in\RR^{k-s\times
s}.\end{equation}

\no By setting
\begin{equation}\label{B1B2}
B_1 = \I_{s1} \P_{s1}^T\O_1 \in\RR^{s\times s}, \qquad B_2 =
\I_{s1} \P_{s2}^T \O_2\in\RR^{s\times k-s},\end{equation}

\no substitution of (\ref{y2_1}) into (\ref{y1}) then provides, at
last, the system of block size $s$ to be actually solved:
\begin{eqnarray}\label{onlys}
F(\bfy_1) &\equiv& \bfy_1 - e\otimes y_0 - h\left[ B_1 \otimes
I_{2m} f(\bfy_1) + \right.\\ &&\left. B_2 \otimes I_{2m}
f\left(\hat u\otimes y_0 +A_1\otimes I_{2m} \bfy_1\right)\right] =
\bf0.\nonumber\end{eqnarray}

By using the simplified Newton method for solving (\ref{onlys}),
and setting
\begin{equation}\label{C} C=B_1+B_2A_1 \in\RR^{s\times s},\end{equation}

\no one obtains the iteration:
\begin{eqnarray}\label{Newt}
 \left( I_s\otimes I_{2m} - hC\otimes J_0 \right) \bfdel^{(n)} &=&
-F(\bfy_1^{(n)}) \equiv \bfpsi_1^{(n)},\\ \bfy_1^{(n+1)} &=& \bfy_1^{(n)} +
\bfdel^{(n)},
\qquad n=0,1,\dots, \nonumber
\end{eqnarray}

\no where $J_0$ is the Jacobian of $f(y)$ evaluated at $y_0$.
Because of the result of Theorem~\ref{mainres}, the following
property of matrix $C$ holds true ~\cite{BIT10}.

\begin{theo}\label{isoC}
The eigenvalues of matrix $C$ in (\ref{C}) coincide with those of
matrix (\ref{Xs}), i.e., with the eigenvalues of the matrix of the
Butcher array of the Gauss-Legendre method of order $2s$.
\end{theo}

\proof Assuming, as usual for simplicity, that the fundamental
stages are the first $s$ ones, one has that the discrete problem
$$\bfy = \pmatrix{c}e\\ u\endpmatrix \otimes y_0 +h A\otimes
I_{2m} f(\bfy),$$

\no which defines the Runge-Kutta formulation of the method, is
equivalent, by virtue of (\ref{y1}), (\ref{y2_1}), (\ref{A1}),
(\ref{B1B2}), to
\begin{eqnarray*}\lefteqn{\pmatrix{cc} I_s & O_{s\times r}\\ -A_1 & I_r
\endpmatrix\otimes
I_{2m} \pmatrix{c} \bfy_1\\ \bfy_2\endpmatrix =}\\&& \pmatrix{c} e\\
\hat u\endpmatrix\otimes y_0 +h\pmatrix{cc} B_1 & B_2\\
O_{r\times s} & O_{r\times r} \endpmatrix\otimes I_{2m} \pmatrix{c} f(\bfy_1)\\
f(\bfy_2)\endpmatrix,\end{eqnarray*}

\no where, as usual, $r=k-s$. Consequently, the eigenvalues of the
matrix $A$ defined in (\ref{AMAT}) coincides with those of the
pencil
\begin{equation}\label{pencil}\left(~\pmatrix{cc} I_s &O_{s\times r}\\ -A_1
& I_r\endpmatrix,~ \pmatrix{cc} B_1 &B_2\\ O_{r\times s} &
O_{r\times r}\endpmatrix~\right).\end{equation}

\no That is,
$$\mu\in\sigma(A) ~~\Leftrightarrow~~ \mu\pmatrix{cc} I_s &O_{s\times r}\\
-A_1 & I_r\endpmatrix\pmatrix{c} \bfu\\ \bfv\endpmatrix =\pmatrix{cc} B_1 &B_2\\
O_{r\times s} & O_{r\times r}\endpmatrix\pmatrix{c} \bfu\\
\bfv\endpmatrix,$$

\no for some nonzero vector $(\bfu^T,\bfv^T)^T$. By setting
$\bfu=\bfo$, one obtains the $r$ zero eigenvalues of the pencil.
For the remaining $s$ (nonzero) ones, it must be $\bfv=A_1\bfu$,
so that:
$$\mu \bfu = \left( B_1\bfu + B_2\bfv \right) = \left( B_1\bfu + B_2A_1\bfu
\right) = C\bfu ~~\Leftrightarrow~~ \mu\in\sigma(C).\QED$$

\medskip
\begin{rem}\label{algo}
 From the result of Theorem~\ref{isoC}, it follows that the spectrum
of $C$ doesn't depend on the choice of the $s$ fundamental
abscissae, within the nodes $\{\tau_i\}$. On the contrary, its
condition number does: the latter appears to be minimized when the
fundamental abscissae are symmetrically distributed and
approximately evenly spaced in the interval $[0,1]$. As a practical
``{\em rule of thumb}'', the following algorithm appears to be
almost optimal:
\begin{enumerate}
\item let the $k$ abscissae $\{\tau_i\}$ be chosen according to a
Gauss-Legendre distribution of $k$ nodes;

\item then, let us consider $s$ equidistributed nodes in $(0,1)$,
say $\{\hat \tau_1, \dots, \hat\tau_s\}$;

\item select, as the fundamental abscissae, those nodes among the $\{\tau_i\}$
which are the closest ones to the $\{\hat \tau_j\}$;

\item define matrix $C$  in (\ref{C}) accordingly.
\end{enumerate}

\no Clearly, for the above algorithm to provide a unique solution
(resulting in a symmetric choice of the fundamental abscissae),
the difference $k-s$ has to be even which, however, can be easily
accomplished.
\end{rem}

In order to give evidence of the effectiveness of the above
algorithm, in Figure~\ref{condC} we plot the condition number of
matrix $C=C(k,s)$, for $s=2,\dots,5$, and $k\ge s$. As one can
see, the condition number of $C(k,s)$ turns out to be nicely
bounded, for increasing values of $k$, which makes the
implementation (that we are going to analyze in the next section)
effective also when finite precision arithmetic is used. For
comparison, in Figure~\ref{condC1}  there is the same plot,
obtained by fixing the fundamental abscissae as the first $s$
ones. In such a case, the condition number of $C(k,s)$ grows very
fast, as $k$ is increased.

\begin{figure}[hp]
\centerline{\includegraphics[width=\textwidth,height=8.5cm]{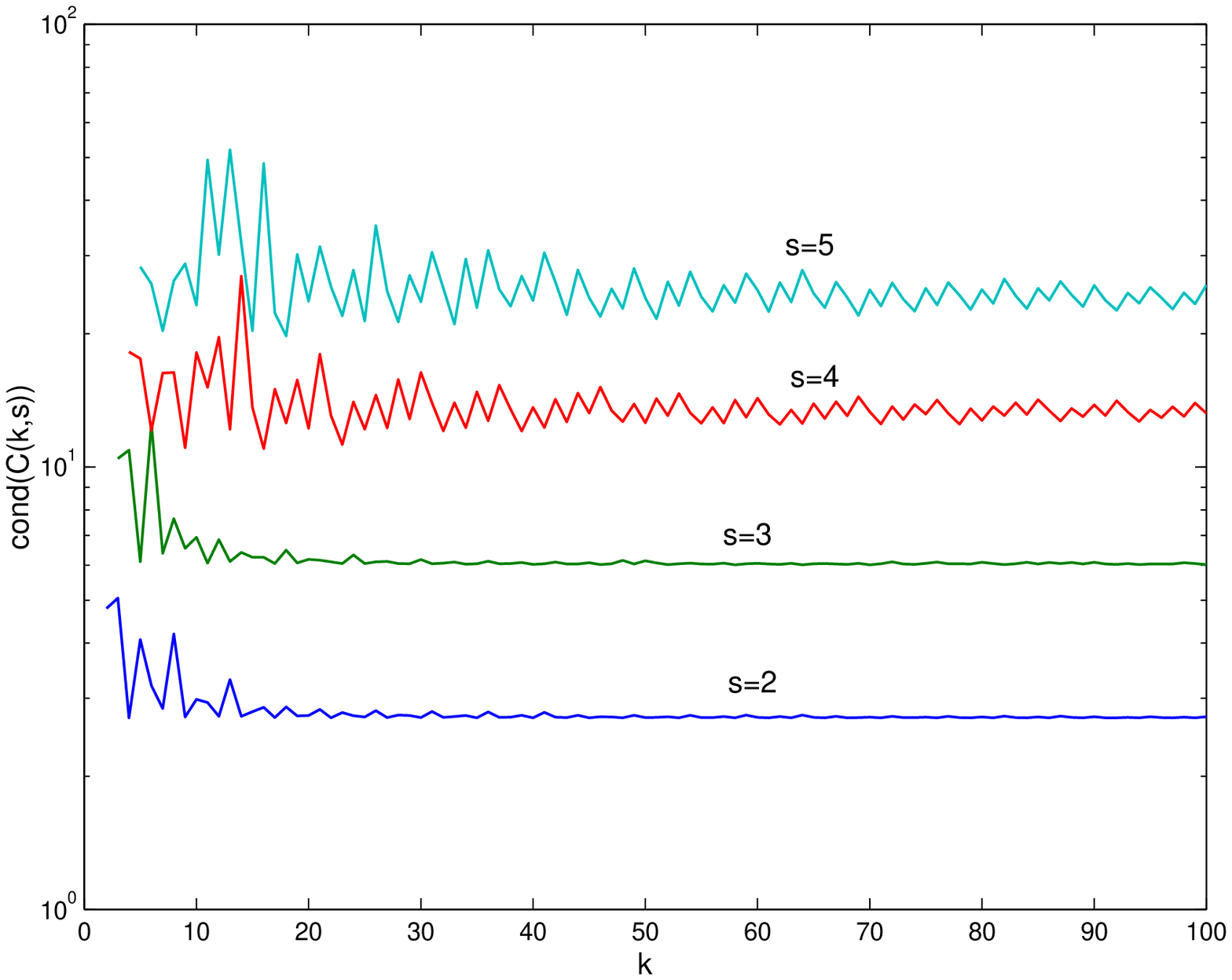}}
\caption{\protect\label{condC} Condition number of the matrix
$C=C(k,s)$, for $s=2,3,4,5$ and $k=s,s+1,\dots,100$, with  the
fundamental abscissae chosen according to the algorithm sketched
in Remark~\ref{algo}.}
\bigskip
\medskip
\centerline{\includegraphics[width=\textwidth,height=8.5cm]{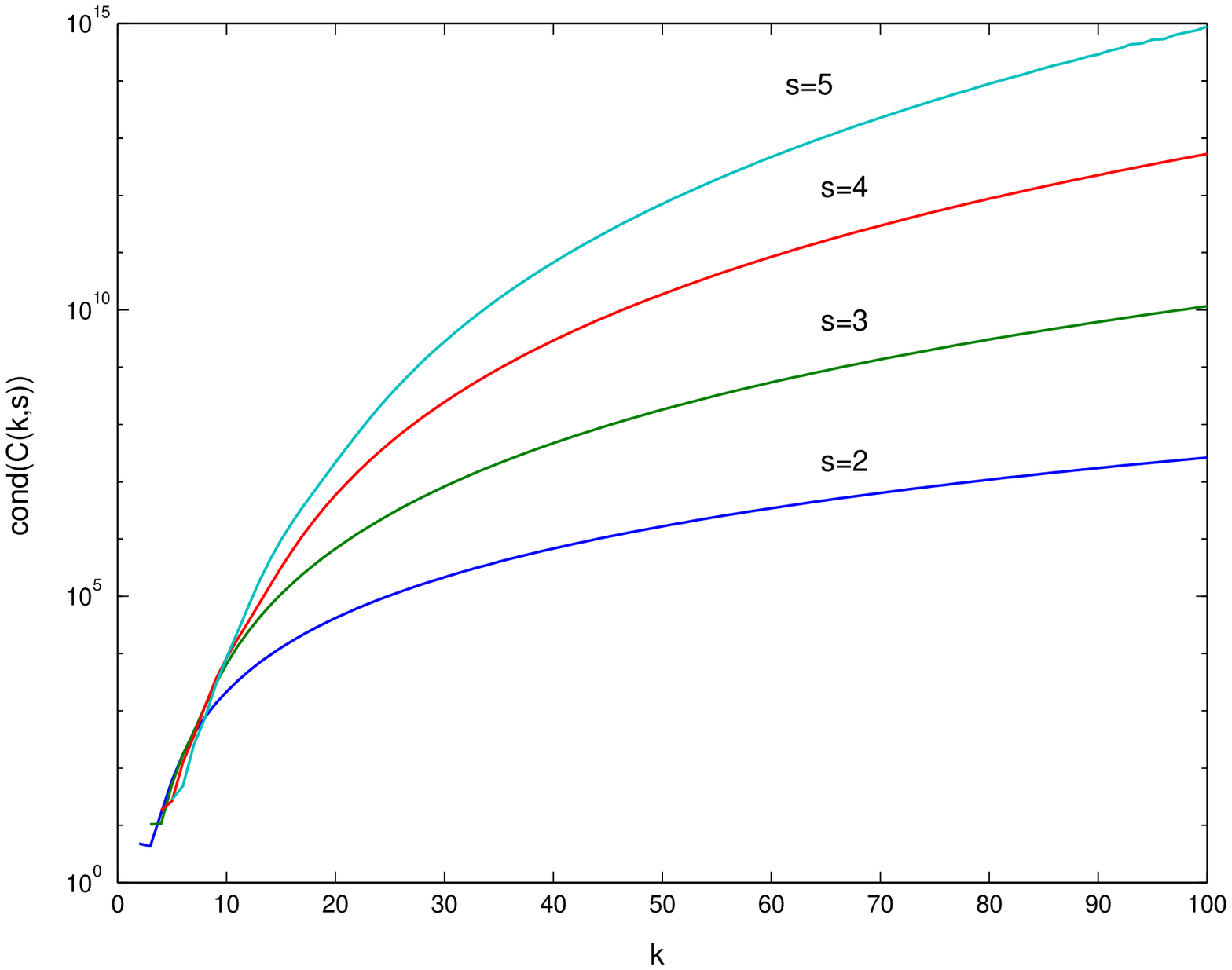}}
\caption{\protect\label{condC1} Condition number of the matrix
$C=C(k,s)$, for $s=2,3,4,5$ and $k=s,s+1,\dots,100$, with the
fundamental  abscissae chosen as the first $s$ ones.}
\end{figure}

\section{Blended implementation}
We observe that, since $C$ is nonsingular, we can recast problem
(\ref{Newt}) in the {\em equivalent form}
\begin{equation}\label{Newt2}
 \gamma\left( C^{-1}\otimes I_{2m} - hI_s\otimes J_0 \right) \bfdel^{(n)} =
-\gamma C^{-1}\otimes I_{2m}\, F(\bfy_1^{(n)}) \equiv \bfpsi_2^{(n)},
\end{equation}

\no where $\gamma>0$ is a free parameter to be chosen later. Let
us now introduce the {\em weight (matrix) function}
\begin{equation}\label{teta}
 \theta = I_s\otimes \Phi^{-1}, \qquad \Phi = I_{2m} -h\gamma
J_0\in\RR^{2m\times 2m},
\end{equation}

\no and the {\em blended formulation} of the system to be solved,
\begin{eqnarray}\nonumber
 M\bfdel^{(n)} &\equiv& \left[ \theta \left(I_s\otimes I_{2m}-hC\otimes
J_0\right) + \right.\\ \nonumber &&\left. (I-\theta)\gamma\left( C^{-1}\otimes
I_{2m}- h I_s\otimes J_0 \right)\right] \bfdel^{(n)}
\\ &=& \theta \bfpsi_1^{(n)}+(I-\theta)\bfpsi_2^{(n)}\equiv
\bfpsi^{(n)}.\label{blendNewt}
\end{eqnarray}

\no The latter system has again the same solution as the previous
ones, since it is obtained as the {\em blending}, with weights
$\theta$ and $(I-\theta)$, of the two equivalent forms
(\ref{Newt}) and (\ref{Newt2}). For iteratively solving
(\ref{blendNewt}), we use the corresponding {\em blended
iteration}, formally given by
\cite{B00,BM02,BM04,BM07,BM08,BM09,BM09a,BMM06,BT2,M04,BIMcode}:
\begin{equation}\label{blendit}\bfdel^{(n,\ell+1)} = \bfdel^{(n,\ell)}
-\theta\left( M\bfdel^{(n,\ell)}-\bfpsi^{(n)}\right), \qquad
\ell=0,1,\dots.\end{equation}

\begin{rem}\label{nonlin}
A nonlinear variant of the iteration (\ref{blendit}) can be
obtained, by starting at $\bfdel^{(n,0)}=\bf 0$ and updating
$\bfpsi^{(n)}$ as soon as a new approximation is available. This
results in the following iteration:
\begin{equation}\label{blendit1}\bfy^{(n+1)} = \bfy^{(n)}
+\theta\bfpsi^{(n)}, \qquad n=0,1,\dots.\end{equation}\end{rem}

\begin{rem}\label{only1} We observe that, for actually performing the
iteration (\ref{teta})--(\ref{blendit}), as well as
(\ref{blendit1}), one has to factor {\em only} the matrix $\Phi$
in (\ref{teta}), which has the same size as that of the continuous
problem.\end{rem}

We end this section by observing that the above iterations
(\ref{blendit}) and (\ref{blendit1}) depend on a free parameter
$\gamma$. It will be chosen in order to optimize the convergence
properties of the iteration,  according to a linear analysis of
convergence, which is sketched in the next section.

\section{Linear analysis of convergence}\label{linear}

The linear analysis of convergence for the iteration
(\ref{blendit}) is carried out by considering the usual scalar
test equation (see, e.g., \cite{BM09} and the references therein),
$$y' = \lam y, \qquad \Re(\lam)<0.$$

\no By setting, as usual $q=h\lam$, the two equivalent
formulations (\ref{Newt}) and (\ref{Newt2}) become, respectively
(omitting, for sake of brevity, the upper index $n$),
$$(I_s-q C) \bfdel = \bfpsi_1, \qquad \gamma( C^{-1} -q I_s) \bfdel =
\bfpsi_2.$$

\no Moreover, \begin{equation}\label{tetaq}\theta =
\theta(q) = (1-\gamma q)^{-1} I_s,\end{equation}

\no and the blended iteration (\ref{blendit}) becomes
\begin{equation}\label{blendq}
\bfdel^{(\ell+1)} = (I_s -\theta(q)M(q))\bfdel^{(\ell)} +
\theta(q)\bfpsi(q),\end{equation}

\no with
\begin{eqnarray}\label{Mq} M(q) &=&
\theta(q)\left(I_s-q C\right) +(I_s-\theta(q))\gamma\left(
C^{-1}-q I_s\right), \\ \bfpsi(q) &=&
\theta(q)\bfpsi_1+(I_s-\theta(q))\bfpsi_2.\nonumber\end{eqnarray}

\no Consequently, the iteration will be convergent if and only if the
spectral radius $\rho(q)$ of the iteration matrix,
\begin{equation}\label{Zq}Z(q) = I_s-\theta(q)M(q),\end{equation}

\no is less than 1. The set
$$\Gamma = \left\{ q\in\CC \,:\, \rho(q)<1 \right\}$$

\no is the {\em region of convergence of the iteration}. The
iteration is said to be:\begin{itemize}
\item {$A$-convergent}, ~ if $\CC^-\subseteq \Gamma$;

\item {$L$-convergent}, ~ if it is $A$-convergent and,
moreover, ~$\rho(q)\rightarrow 0$,~ as ~$q\rightarrow\infty$.
\end{itemize}

\begin{table}[t]
\caption{\protect\label{params} optimal values (\ref{gammaopt}),
and corresponding maximum amplification factors (\ref{rostarmin}),
for various values of $s$.} \centerline{\begin{tabular}{|r|r|r|}
\hline
$s$ & $\gamma$ & $\rho^*$\\
\hline
2 &0.2887 &0.1340\\
3 &0.1967 &0.2765\\
4 &0.1475 &0.3793\\
5 &0.1173 &0.4544\\
6 &0.0971 &0.5114\\
7 &0.0827 &0.5561\\
8 &0.0718 &0.5921\\
9 &0.0635 &0.6218\\
10 &0.0568 &0.6467\\
\hline
\end{tabular}}
\end{table}

\no For the iteration (\ref{blendq}) one verifies that (see
(\ref{tetaq}), (\ref{Mq}), and (\ref{Zq}))
\begin{equation}\label{Zq1}
Z(q) = \frac{q}{(1-\gamma q)^2}C^{-1}\left(C-\gamma I_s\right)^2,
\end{equation}

\no which is the null matrix at $q=0$ and at $\infty$.
Consequently, the iteration will be $A$-convergent (and,
therefore, $L$-convergent), provided that {\em maximum
amplification factor},
\begin{equation}\label{rostar}
\rho^* \equiv \max_{\Re(q)=0} \rho(q) ~\le 1.
\end{equation}

\no From (\ref{Zq1}) one has that, by setting hereafter
$\sigma(C)$ the spectrum of matrix $C$,
$$\mu\in\sigma(C)
~\Leftrightarrow~\frac{q(\mu-\gamma)^2}{\mu(1-\gamma
q)^2}\in\sigma(Z(q)).$$

\no By taking into account that
$$\max_{\Re(q)=0}\frac{|q|}{|(1-\gamma q)^2|} =
\frac{1}{2\gamma},$$

\no one then obtains that
$$\rho^* = \max_{\mu\in\sigma(C)}
\frac{|\mu-\gamma|^2}{2\gamma|\mu|},$$

\no For Gauss-Legendre methods (and, then, for any matrix $C$
having the same spectrum), it can be shown that (see
\cite{BM02,BMM06}) the choice
\begin{equation}\label{gammaopt} \gamma = |\mu_{\min}|\equiv
\min_{\mu\in\sigma(C)}|\mu|,\end{equation}

\no minimizes $\rho^*$, which turns out to be given by
\begin{equation}\label{rostarmin} \rho^* = 1 -\cos \phi_{\min}
~<1, \qquad \phi_{\min}={\rm Arg}(\mu_{\min}). \end{equation}

In Table~\ref{params}, we list the optimal value of the parameter
$\gamma$, along with the corresponding maximum amplification
factor $\rho^*$, for various values of $s$, which confirm that the
iteration (\ref{blendq}) is $L$-convergent.

\begin{rem}
We then conclude that the {\em blended iteration} (\ref{blendit})
turns out to be $L$-convergent, for any HBVM$(k,s)$ method, for
all $s\ge1$ and $k\ge s$.
\end{rem}

We end this chapter, by emphasizing that the property of
$L$-convergence has proved to be computationally very effective,
as testified by the successful implementation of the codes {\tt
BiM} and {\tt BiMD} \cite{M04,BIMcode}. We then expect good
performances also for the {\em blended implementation} of
HBVM$(k,s)$.

\chapter{Notes and References}\label{chap6}

The approach of using discrete line integrals has been used, at
first, by Iavernaro and Trigiante, in connection with the study of
the properties of the trapezoidal rule \cite{IT0,IT1,IT2}.

It has been then extended by Iavernaro and Pace \cite{IP1}, thus
providing the first example of conservative methods, basically an
extension of the trapezoidal rule, named {\em $s$-stage
trapezoidal methods}: this is a family of energy-preserving
methods of order 2, able to preserve polynomial Hamiltonian
functions of arbitrarily high degree.

Later generalizations allowed Iavernaro and Pace \cite{IP2}, and
then Iavernaro and Trigiante \cite{IT3}, to derive energy
preserving methods of higher order.

The general approach, involving the shifted Legendre polynomial
basis, which has allowed a full complete analysis of HBVMs, has
been introduced in \cite{brugnano09bit} (see also \cite{BIT0})
and, subsequently, developed in \cite{BIT09}.

The Runge-Kutta formulation of HBVMs, along with their connections
with collocation methods, has been studied in \cite{BIT10_1}.

The isospectral property of HBVMs has been also studied in
\cite{BIT10}, where the {\em blended} implementation of the
methods has been also introduced.

Computational aspects, concerning both the computational cost and
the efficient numerical implementation of HBVMs, have been studied
in \cite{BIS} and \cite{BIT10}.

Relevant examples have been collected in \cite{BIS1}, where the
potentialities of HBVMs are clearly outlined, also demonstrating
their effectiveness with respect to standard symmetric and
symplectic methods.

Blended implicit methods have been studied in a series of papers
\cite{B00,BM02,BM04,BM07,BM08,BM09,BM09a,BMM06,M04} and have been
implemented in the two computational codes {\tt BiM} and {\tt
BiMD} \cite{BIMcode}.


\begin{thebibliography}{99}
\setlength{\itemsep}{0cm}

\bibitem{AS} M.\,Abramovitz, I.A.\,Stegun. {\em Handbook of Mathematical
Functions}. Dover, 1965.

\bibitem{B00} L.\,Brugnano. Blended block BVMs (B$_3$VMs): A family of
economical implicit methods for ODEs. {\em J. Comput. Appl.Math.}
{\bf 116} (2000) 41--62.

\bibitem{BIS} L.\,Brugnano, F.\,Iavernaro, T.\,Susca. Hamiltonian BVMs (HBVMs):
implementation details and applications. ``Proceedings of ICNAAM
2009'', {\em AIP Conf. Proc.} {\bf 1168} (2009) 723--726.

\bibitem{BIS1} L.\,Brugnano, F.\,Iavernaro, T.\,Susca. Numerical
comparisons between Gauss-Legendre methods and Hamiltonian BVMs
defined over Gauss points. {\em Monograf\'\i as de la Real
Academia de Ciencias de Zaragoza}, Special Issue devoted to the
65th birthday of Manuel Calvo,  (Submitted) 2010 ({\tt
arXiv:1002.2727}).

\bibitem{BIT0} L.\,Brugnano, F.\,Iavernaro, D.\,Trigiante. Hamiltonian BVMs (HBVMs):
a family of ``drift-free'' methods for integrating polynomial
Hamiltonian systems. ``Proceedings of ICNAAM 2009'', {\em AIP
Conf. Proc.} {\bf 1168} (2009) 715--718.

\bibitem{brugnano09bit} L.\,Brugnano, F.\,Iavernaro,
D.\,Trigiante. Analisys of Hamiltonian Boundary Value Methods
(HBVMs) for the numerical solution of polynomial Hamiltonian
dynamical systems. {\em BIT}, submitted for publication (2009)
({\tt arXiv:0909.5659}).

\bibitem{BIT09} L.\,Brugnano, F.\,Iavernaro,
D.\,Trigiante. Hamiltonian Boundary Value Methods (Energy
Conserving Discrete Line Integral Methods). {\em Jour. Numer.
Anal., Industrial and Appl. Math.}, submitted for publication
(2009) ({\tt arXiv:0910.3621}).

\bibitem{BIT10} L.\,Brugnano, F.\,Iavernaro,
D.\,Trigiante. Isospectral Property of HBVMs and their Blended
Implementation. {\em BIT}, submitted for publication (2010) ({\tt
arXiv:1002.1387}).

\bibitem{BIT10_1} L.\,Brugnano, F.\,Iavernaro, D.\,Trigiante. Isospectral
Property of HBVMs and their connections with Runge-Kutta
collocation methods. {\em Preprint}, 2010 ({\tt arxiv:1002.4394}).

\bibitem{BM02} L.\,Brugnano, C.\,Magherini. Blended implementation of block
implicit methods for ODEs. {\em Appl. Numer. Math.} {\bf 42} (2002) 29--45.

\bibitem{BM04} L.\,Brugnano, C.\,Magherini. The {\tt BiM} code for the numerical
solution of ODEs. {\em J. Comput. Appl. Math.} {\bf 164--165} (2004) 145--158.

\bibitem{BM07} L.\,Brugnano, C.\,Magherini. Blended implicit methods for solving
ODE and DAE problems, and their extension for second order problems. {\em J.
Comput. Appl. Math.} {\bf 205} (2007) 777--790.

\bibitem{BM08} L.\,Brugnano, C.\,Magherini. Blended General Linear Methods based
on Generalized BDF. {\em AIP Conf. Proc.} {\bf 1048} (2008) 871--874.

\bibitem{BM09} L.\,Brugnano, C.\,Magherini. Recent Advances in Linear Analysis
of Convergence for Splittings for Solving ODE problems. {\em Appl. Numer. Math.}
{\bf 59} (2009) 542--557.

\bibitem{BM09a} L.\,Brugnano, C.\,Magherini. Blended General Linear Methods
based on Boundary Value Methods in the GBDF family. {\em Journal of
Numerical Analysis, Industrial and Applied Mathematics} {\bf 4},
1-2 (2009) 23--40.

\bibitem{BMM06} L\, Brugnano, C.\,Magherini, F.\,Mugnai. Blended implicit
methods for the numerical solution of DAE problems. {\em J. Comput. Appl.
Math.} {\bf 189} (2006) 34--50.

\bibitem{BT} L.\,Brugnano, D.\,Trigiante. {\em Solving
Differential Problems by Multistep Initial and Boundary Value
Methods}. Gordon and Breach, Amsterdam, 1998.

\bibitem{BT2} L.\,Brugnano, D.\,Trigiante. Block implicit methods for ODEs, in:
D. Trigiante (Ed.), {\em Recent Trends in Numerical Analysis}.
Nova Science Publ. Inc., New York, 2001, pp. 81--105.

\bibitem{BT09} L.\,Brugnano, D.\,Trigiante. Energy drift in the numerical
integration of Hamiltonian problems. {\em Journal of Numerical
Analysis, Industrial and Applied Mathematics} (to appear).

\bibitem{Faou} E.\,Faou, E.\,Hairer, T.-L.\,Pham. Energy conservation with non-symplectic
methods: examples and counter-examples. {\em BIT Numerical
Mathematics} {\bf 44} (2004) 699--709.

\bibitem{hairer06gni} E.\,Hairer, C.\,Lubich, G.\,Wanner. {\em Geometric
Numerical Integration. Structure-Preserving Algorithms for
Ordinary Differential Equations, 2$^{nd}$ ed.}, Springer, Berlin,
2006.

\bibitem{HNW} E.\,Hairer, G.\,Wanner. {\em Solving Ordinary Differential
Equations I, 2nd ed.}, Springer, Berlin, 2000.

\bibitem{HW} E.\,Hairer, G.\,Wanner. {\em Solving Ordinary
Differential Equations II}, Springer, Berlin, 1991.

\bibitem{IP1} F.\,Iavernaro, B.\,Pace. $s$-Stage Trapezoidal Methods for the
Conservation of Hamiltonian Functions of Polynomial Type. {\em AIP
Conf. Proc.} {\bf 936} (2007) 603--606.

\bibitem{IP2} F.\,Iavernaro, B.\,Pace. Conservative Block-Boundary
Value Methods for the Solution of Polynomial Hamiltonian Systems.
{\em AIP Conf. Proc.} {\bf 1048} (2008) 888--891.

\bibitem{IT0} F.\,Iavernaro, D.\,Trigiante. On some conservation properties of
the Trapezoidal Method applied to Hamiltonian systems. {\em ICNAAM 2005
proceedings}, T.E.\,Simos, G.\,Psihoyios, Ch.\,Tsitouras (Eds.). Wiley-VCH,
Weinheim, 2005, pp. 254--257 (ISBN:3527406522).

\bibitem{IT1} F.\,Iavernaro, D.\,Trigiante. Discrete conservative vector
fields induced by the trapezoidal method. {\em J. Numer. Anal.
Ind. Appl. Math.} {\bf 1} (2006) 113--130.

\bibitem{IT2} F.\,Iavernaro, D.\,Trigiante. State-dependent symplecticity
and area preserving numerical methods. {\em J. Comput. Appl.
Math.} {\bf 205} no.\,2  (2007) 814--825.

\bibitem{IT3} F.\,Iavernaro, D.\,Trigiante. High-order symmetric schemes for
the energy conservation of polynomial Hamiltonian problems. {\em
J. Numer. Anal. Ind. Appl. Math.} {\bf 4},1-2 (2009) 87--101.

\bibitem{M04} C.\,Magherini. {\em Numerical Solution of Stiff ODE-IVPs via
Blended Implicit Methods: Theory and Numerics.} PhD thesis,
Dipartimento di Matematica ``U.\,Dini'', Universit\`a degli Studi
di Firenze, September 2004 (Available at the url \cite{BIMcode}).

\bibitem{James} J.D.\,Mireles James. Celestial mechanics notes,
Set~1: Introduction to the \mbox{$N$-Body Problem}. {\em Available
at url:}\\ {\tt
http://www.math.utexas.edu/users/jjames/celestMech}

\bibitem{BIMcode} Codes {\tt BiM}/{\tt BiMD} Homepage:\\ {\tt
http://www.math.unifi.it/\~{}brugnano/BiM/index.html}

\end{thebibliography}
\end{document}